\documentclass[12pt]{article}
\usepackage{fancyhdr}
\usepackage{amssymb}
\usepackage{amsmath}
\usepackage{amsthm}
\usepackage{graphicx}
\usepackage{psfrag}
\usepackage{graphics}
\usepackage{amsthm}
\usepackage{hyperref,enumitem}
\usepackage[dvipsnames]{xcolor}
\usepackage[english]{babel}

\newcommand{\N}{\mathbb N}

\newcommand{\var}{\mathrm{var}}

\newcommand{\R}{\mathbb{R}}
\newcommand{\Ss}{\mathbb{S}}
\newcommand{\Q}{\mathbb{Q}}
\newcommand{\Z}{\mathbb Z}

\newcommand{\Diff}{\mathrm{Diff}}

\newcommand{\Ddz}{\mathrm{Diff}^\infty_{\Delta,0}}

\newcommand{\eps}{\varepsilon}

\theoremstyle{theorem}

\newtheorem{thm}{Theorem}[section]
\newtheorem{thmintro}{Theorem}

\newtheorem*{main}{Main Theorem}
\newtheorem{prop}[thm]{Proposition}

\newtheorem{lem}[thm]{Lemma}

\newtheorem{claim}[thm]{Claim}

\newtheorem{defn}[thm]{Definition}
\theoremstyle{remark}
\newtheorem{rem}[thm]{Remark}

\newcommand{\id}{\mathrm{id}}
\newcommand{\f}{\varphi}

\begin{document}

\date{}
\vspace{-1cm}

\date{}
\author{H\'el\`ene Eynard-Bontemps \, \& \, Emmanuel Militon}

\title{Almost reducibility, distortion and local perfection for diffeomorphisms of one-manifolds }

\maketitle

\begin{abstract}
In this article, we characterize the distortion elements of the group of smooth diffeomorphisms of the circle and of the group of compactly supported smooth diffeomorphisms of the real line. More precisely, we prove that, in this context, an element is distorted if and only if it is \emph{almost reducible}, that is if and only if 
it has conjugates arbitrarily close to an isometry. For diffeomorphisms with fixed points, we show that this is equivalent to being the time-$1$ map of a $C^1$ vector field without hyperbolic zero. The equivalence between distortion and almost reducibility relies on new more general results about distortion elements in groups of diffeomorphisms of manifolds and on a new local perfection result for the group of compactly supported smooth diffeomorphisms of the real line.
\end{abstract}

%\textcolor{red}{Je suis un peu benête, je comprends pas comment marche la biblio (il y a un fichier bib et un fichier bbl, visiblement c'est le bbl qui contient ce qu'il faut, pourtant à la 
%fin, c'est le fichier Biblio qu'on appelle...}
%
%{\color{blue} C'est bien le fichier biblio qu'il faut modifier (je n'ai jamais touché au fichier bbl, j'imagine qu'il est modifié quand tu compiles en bibtex) en ajoutant les nouvelles références. Quand tu compiles l'article avec la nouvelle référence, tu compiles une fois en pdflatex, une fois en bibtex, puis deux fois en pdflatex.
%Pour rajouter une référence au fichier biblio, je copie colle la référence bibtex donnée directement par zentralblatt ou mathscinet.
%}
%
%à introduire
  The notion of distortion element in a group is a geometric group theoretic notion which was introduced by Gromov. 
%\marginpar{\tiny\color{red}Andrés corrige souvent mes ``which'' par des ``that'', à vérifier...}
{In a finitely generated group $G$, an element $g$ is said to be distorted if the wordlength $\ell_S(g^n)$ of the powers $g^{n}$, with respect to some finite generating set $S$ of $G$, is sublinear in $n$. Note that, for any $p,q \geq 0$, the inequality $\ell_S(g^{p+q}) \leq \ell_S(g^{p})+\ell_S(g^{q})$ holds so that the sequence $\left(\frac{\ell_S(g^{n})}{n}\right)_n$ always converges.} The notion of distorted element is independent of the chosen generating set and amounts to saying, in the geometric group theory language, that the corresponding morphism from the group $\mathbb{Z}$ to the group $G$ is not quasi-isometric. This notion is stable under group morphisms.

Distorted elements naturally appear in some Baumslag-Solitar groups and in the Heisenberg group. Recall that a solvable Baumslag-Solitar group is a group $BS(k)$, for an integer $k$, given by the presentation
$$BS(k)=<a,b \ | \ aba^{-1}=b^{k}>.$$
When $|k|\geq 2$, the element $b$ is distorted in this group: the element $b^{k^{n}}= a^{n}ba^{-n}$ has a wordlength smaller than or equal to $2n+1$ relative to the generating set $\left\{a,b\right\}$ and $\frac{2n+1}{k^n}$ tends to $0$ as $n$ tends to $+\infty$. Hence, in a more general group, an element which is conjugate to a power of itself different from $1$ or $-1$ is distorted.

The integral Heisenberg group is the group $H$ of upper triangular integral matrices whose diagonal coefficients are all equal to $1$. In this group, we let
$$a =\begin{pmatrix}
1 & 1 & 0 \\ 0 & 1 & 0 \\ 0 & 0 & 1
\end{pmatrix}, b =\begin{pmatrix}
1 & 0 & 0 \\ 0 & 1 & 1 \\ 0 & 0 & 1
\end{pmatrix}, c =\begin{pmatrix}
1 & 0 & 1 \\ 0 & 1 & 0 \\ 0 & 0 & 1
\end{pmatrix}.
$$
The elements $a$ and $b$ generate the group $H$ and commute with $c=[a,b]$. Moreover, for any integer $n$, $c^{n^2}=[a^n,b^n]=a^n b^n a^{-n}b^{-n}$ has a word-length smaller than or equal to $4n$. Hence the element $c$ is distorted in $H$.

Polterovich in \cite{Pol} and, later, Franks and Handel in \cite{FH}, had the idea to study this notion in non-finitely generated groups, namely groups of homeomorphisms or diffeomorphisms of manifolds. In a more general group, an element is said to be distorted if it lies in a finitely generated subgroup in which it is distorted. Using this notion, Franks and Handel, generalizing the earlier result by Polterovich, proved that actions of nonuniform lattices of some semisimple Lie groups had to act almost trivially by $C^{1}$ diffeomorphisms which preserve a measure with infinite support on surfaces with positive genus. In their article, Franks and Handel also propose a more systematical study of distortion elements in groups of transformations of manifolds.

Since then, various authors have studied distortion elements in groups of homeomorphisms or diffeomorphisms of manifolds. For instance, Calegari and Freedman proved in \cite{CF} that any element in the group of homeomorphisms of a sphere is distorted. Avila proved in \cite{Av08} that rotations are distorted in the group of smooth diffeomorphisms of the circle.\medskip

%\marginpar{\tiny\color{red}``in this article'' me paraît un peu redondant avec ``in what follows''. Peut-être plutôt ``In this context, we manage to give a complete description of...''}
 In what follows, we restrict our attention to the group $G=\mathrm{Diff}^{\infty}_+(\mathbb{S}^1)$ of smooth  orientation-preserving diffeomorphisms of the circle $\mathbb{S}^1$ which we identify to $\mathbb{R}/\mathbb{Z}$  or to the group $G=\mathrm{Diff}^{\infty}_c(\mathbb{R})$ of smooth compactly supported diffeomorphisms of the real line $\mathbb{R}$, that is the group of smooth diffeomorphisms which fix all the points outside of a compact subset. Throughout the article, by ``smooth'', we will mean $C^\infty$. In this context, we manage to give a complete description of the distortion elements in the group $G$.  
 
 %{\color{red} Avant de se lancer dans le cas différentiable, ne faudrait-il pas dire qu'en $C^0$ tout le monde est distordu et pourquoi ?}

 As a consequence of Calegari and Freedman's results and techniques, any element in the group of homeomorphisms of the circle and of the group of compactly supported homeomorphism of the real line is distorted. More elementarily, observe that any compactly supported homeomorphism of the real line is conjugated to its square, which implies that it is distorted.   However, in higher regularity, there are obstructions to being distorted.

To prove that an element $g$ in a group $G$ is undistorted, it suffices to find a length, that is a map $\ell: G \rightarrow [0,+\infty)$ with $\ell(h_1 h_2)\leq \ell(h_1)+\ell(h_2)$ for any $h_1,h_2\in G$, such that $\frac{\ell(g^{n})}{n}$ has a nonzero limit when $n \rightarrow +\infty$. Indeed, in a finitely generated group with wordlength $\ell_S$, we have that, for any length $\ell$, there exists $C>0$ such that $\ell \leq C \ell_S$, which implies the property.

In the group $G=\mathrm{Diff}^{\infty}_+(\mathbb{S}^1)$ or $\mathrm{Diff}^{\infty}_c(\mathbb{R})$, such a length is provided by the supremum of the absolute value of the logarithm of the derivative, by the chain rule. This implies that a smooth diffeomorphism $g$ with a hyperbolic fixed point, that is a point $a$ where its derivative  $Dg(a)$ is different from $1$, is undistorted, as the derivative of $g^{n}$ at $a$ is $(Dg(a))^{n}$. Hence distorted elements have no hyperbolic fixed point. As a consequence, distorted elements have no hyperbolic periodic orbit, as such elements have a power with a hyperbolic fixed point.

There is yet a more subtle obstruction to being distorted for smooth diffeomorphisms, called the asymptotic variation. This obstruction was introduced by Navas in \cite{Na23} and used in \cite{Na21} to provide an example of a $C^{2}$-diffeomorphism of $[0,1]$ which is distorted in the group of $C^1$ diffeomorphisms but undistorted in the group of $C^2$ diffeomorphisms. For a map $\varphi:\mathbb{R} \rightarrow \mathbb{R}$, we define its total variation
$$\mathrm{Var}(\varphi)=\sup \left\{ \sum_{i=0}^{\ell-1} |\varphi(x_{i+1})-\varphi(x_i)| \ | \ x_0<x_1<\ldots<x_{\ell} \right\} \in [0,+\infty].$$
If the map $\varphi$ has a compact support and is $C^1$, it is known that this quantity is finite and equal to $\int_{\mathbb{R}} |D\varphi|(x)dx$. We can define a similar quantity $\mathrm{Var}(\varphi)$ for maps $\varphi: \mathbb{S}^1 \rightarrow \mathbb{R}$. 
%\marginpar{\tiny\color{red} peut-être définir la variation asymptotique dans une phrase à part, sans parler de distorsion}
The map which, to any element $f$ in our group $G$, associates $\mathrm{Var}\left( \log(Df) \right)$ is a length function on $G$, by the chain rule.

\noindent Hence for any element $f\in G$, one can define its \textit{asymptotic variation} by
$$\mathrm{V}_{\infty}(f)=\lim_{n \rightarrow +\infty} \frac{\mathrm{Var}\left( \log(Df^n) \right)}{n}\quad\in\R_+.$$
Note that the vanishing of this quantity is well understood: for $C^2$ circle diffeomorphisms with irrational rotation number, it is automatic according to \cite{Na23}, while for $C^2$ diffeomorphisms with fixed points, the first author and Navas proved in \cite{EN21} that it is equivalent to being the time-$1$ map of the flow of a $C^1$ vector field without hyperbolic zero.   

 According to the general discussion above, the nonvanishing of the asymptotic variation is an obstruction to distortion. It also provides a natural obstruction to another property we investigate in this article, often called \emph{almost reducibility}.
%, and the main result of this article is that it is in fact the only one when one restricts to diffeomorphisms with fixed/periodic points.
%\marginpar{\tiny {\color{red}je mettrais ``periodic'' au lieu de ``fixed'' et du coup je changerais un peu la suite}}

%In this article, we also relate distortion elements to almost reducible elements.

\begin{defn} 
\label{d:reduc}
An element $f\in G$
%or the interval 
is said to be \emph{almost reducible} if it can be brought arbitrarily close to an isometry by conjugacy in $G$ by elements with a common compact support.
%there exists a sequence of smooth conjugants $(\varphi_n)_{n\in\N}$ such that $\f_n f \f_n^{-1}$ goes to the identity when $n$ goes to infinity.
\end{defn}

Observe that, in the case $G=\mathrm{Diff}^{\infty}_c(\mathbb{R})$, this amounts to saying that the element can be brought arbitrarily close to the identity by conjugation by elements with a common compact support.

Note that $\mathrm{Var}\left( \log(Dg_n) \right)$ goes to $0$ for any sequence $(g_n)$ of elements of $G$ with a common compact support which tends to an isometry. Furthermore, one easily sees that $\mathrm{V}_{\infty}$ is an invariant of smooth conjugacy and that $\mathrm{V}_{\infty}(g)\le \mathrm{Var}\left( \log(Dg) \right)$ for any $g\in G$, so in fact 
$$\forall f\in G,\quad\mathrm{V}_{\infty}(f)\le \inf_{h\in G} \mathrm{Var}\left( \log(D(hfh^{-1}) \right).$$
This shows that the asymptotic variation of a quasi reducible diffeomorphism must vanish. 

Our main theorem claims that the asymptotic variation is actually the only obstruction both to quasi reducibility and to distortion, for elements with a periodic orbit (for the general case, see Remark \ref{r:AK}).

\begin{main} \label{t:completedescriptiondistortion}
Let $g$ be an element of $G=\mathrm{Diff}^{\infty}_c(\mathbb{R})$ or $\mathrm{Diff}^{\infty}_+(\mathbb{S}^1)$ with at least one periodic orbit. The following statements are equivalent.
\begin{enumerate}
\item The element $g$ is distorted in $G$.
\item The asymptotic variation of $g$ vanishes.
\item The element $g$ is almost reducible in $G$.
\end{enumerate}
Furthermore, the implications $1. \Rightarrow 2.$ and $3. \Rightarrow 1.$ hold without the periodic orbit hypothesis.
\end{main}

We explained above why the implications $1. \Rightarrow 2.$ and $3. \Rightarrow 2.$ hold.

%\marginpar{\tiny\textcolor{red}{Comparer EBN24 pour les cas $C^1$ et $C^2$ ? Ou plus tard quand on parlera de la stratégie, qui ne marche pas en régularité finie ?} {\color{blue}Oui, je pense que c'est bien de dire quelquechose sur les régularités intermédiaires ici et juste signaler, pour le moment, que notre résultat n'est valable qu'en $C^{\infty}$. } }

%\marginpar{\tiny {\color{blue}Comme on vient de rappeler ce résultat juste au-dessus, je te propose de supprimer ce paragraphe.}}
%Recall that, by Theorem A of \cite{EN21}, the asymptotic variation of a diffeomorphism vanishes if and only if all its fixed points are parabolic, and the diffeomorphism is the time $1$ of the flow of a $C^1$ vector field.

When $G=\mathrm{Diff}^{\infty}_c(\mathbb{R})$, as any element of this group has a fixed point, we obtain a complete characterization of the distorted elements in the group $G$.

As a diffeomorphism of the circle with rational rotation number has a periodic orbit, we deduce immediately from the theorem that any smooth diffeomorphism of the circle with rational rotation number is distorted if and only if its asymptotic variation vanishes if and only if it is almost reducible.

\begin{rem}
\label{r:AK}
By a still unwritten claim by Avila and Krikorian, any smooth diffeomorphism of the circle with irrational rotation number is almost reducible. Using this claim, we obtain that any smooth diffeomorphism of the circle with irrational rotation number, and thus vanishing asymptotic variation by \cite{Na21}, is distorted. Hence, taking this claim for granted,
%\marginpar{\tiny\color{red}j'enlèverais le vert} 
The Main Theorem 
%\ref{t:completedescriptiondistortion} 
still holds without the periodic point hypothesis and we obtain a complete characterization of the distorted elements in the group $\mathrm{Diff}^{\infty}_+(\mathbb{S}^1)$.
\end{rem}

The $C^\infty$ and $C^0$ settings are the only ones where such a complete characterization is known. In Remark \ref{r:proj-dist} below, we discuss the case of intermediate regularity. 

%\medskip
%
%\emph{\color{red} An element of $G$ is distorted \emph{if and only if} its asymptotic variation vanishes.}
%\medskip

%\marginpar{\tiny {\color{blue} Parler du lien avec la notion de distorsion de Rosendal pour les groupes polonais ? En dimension 1, les deux notions coïncident.} \color{red}{Comme conséquence de notre thm tu veux dire ? {\color{blue} Oui.} Il y a aussi l'article de Cohen, à discuter oralement}}

The main result above is the consequence of two theorems that we introduce now. More precisely, the first one contains the implication $2. \Rightarrow 3.$, and the second one the implication $3. \Rightarrow 1.$.

%\textcolor{blue}{Est-ce que la définition suivante ne serait pas plutôt la définition de "almost reducibility" plutôt que "reducibility" ? Il me semble que la reducibility, c'est le cas où on est exactement conjugué à une rotation dans le cas des difféos du cercle. Mais tu connais sûrement mieux les usages que moi.}
%\textcolor{red}{En effet, j'ai plutôt vu ``almost reducible'' ou ``quasi reducible''... Krikorian utilise ``almost reducible'', on n'a qu'à faire comme lui ;)}
%\textcolor{blue}{J'ai changé les "reducible" en "almost reducible"}

\begin{thmintro}
\label{t:reduc2}
Any element of $\Diff^\infty_+(\Ss^1)$, $\Diff^\infty_+([0,1])$ or $\Diff^\infty_c(\R)$ with at least one periodic orbit
%or the interval 
and with vanishing asymptotic variation is almost reducible. % Furthermore one can impose that the conjugants bringing it closer and closer to the identity are all supported in a common compact subset. 
%Furthermore, one can impose that the conjugants bringing it closer and closer to the identity coincide with a power of 
\end{thmintro}

%{\color{blue} Our proof of Theorem \ref{t:reduc2} is specific to the $C^{\infty}$ setting.}

%We combine the above theorem with the following to obtain the main theorem. 
%Let $G=\mathrm{Diff}^\infty_+(\Ss^1)$ or $\Diff^\infty_c(\R)$.

\begin{thmintro} \label{Thm:conjugacyanddistortion}
Let $f \in G=\mathrm{Diff}^\infty_+(\Ss^1)$ or $\Diff^\infty_c(\R)$ and $g$ be a distorted element of $G$. Suppose that there exists a sequence $(h_n)_{n \geq 0}$ consisting of diffeomorphisms in $G$ with a common compact support such that the sequence of conjugates $(h_n f h_n^{-1})_n$ converges to $g$ in the $C^{\infty}$ sense. Then $f$ is distorted in $G$.
\end{thmintro}

% Avila and Krikorian claimed that any smooth diffeomorphism of the circle with irrational rotation number is almost reducible. Since rotations are distorted according to \cite{Av08}, this claim and  Theorem \ref{Thm:conjugacyanddistortion} imply that any smooth diffeomorphism with irrational rotation number is distorted.\medskip

Theorem \ref{t:reduc2} is proved in \S\ref{s:reduc} and partly relies on classical works of Szekeres, Kopell and Mather \cite{Sz58,Ko,Mather} regarding conjugacy classes of interval diffeomorphisms, which are introduced in \S\ref{s:classical}. 
 A central ingredient of the proof of Theorem \ref{Thm:conjugacyanddistortion} is the \emph{local fragmented perfection} of the topological groups under scrutiny. This is the only part in the proof of Theorem \ref{Thm:conjugacyanddistortion} which is specific to the $C^{\infty}$ category, as we are unable to prove similar results for $C^{r}$-diffeomorphisms, with $r$ finite.
For $\mathrm{Diff}^\infty_+(\Ss^1)$, this property was proved by Avila \cite{Av08}. In \S\ref{s:loc-perf}, we present Avila's proof and extend it to $\Diff^\infty_c(\R)$. We then prove Theorem \ref{Thm:conjugacyanddistortion} in \S\ref{s:dist-reduc}. This strategy can actually be generalized to any manifold and \S\ref{s:extension} presents the corresponding extension.

\begin{rem}
\label{r:proj-dist}
In the $C^1$ setting, we saw, using the length function $f\mapsto \|\log Df\|_\infty$, that a distorted element cannot have hyperbolic periodic points, and one easily sees that the same holds for almost reducible diffeomorphisms. Navas proved in \cite{Na11} that conversely, the absence of hyperbolic periodic points implies almost reducibility (in the $C^1$ topology). It is unknown, however, whether it also implies distortion. We only know from \cite{EBN} that among $C^1$ diffeomorphisms without hyperbolic periodic points, those that are distorted form a dense subset.
 
Growing slightly in regularity, the natural setting to define the asymptotic variation, based on the length function $f\mapsto \var(\log Df)$, is that of $C^{1+bv}$ or $C^{1+ac}$ diffeomorphisms, i.e. $C^1$ diffeomorphisms whose derivative has bounded variation or is absolutely continuous. In this context, Navas proved in \cite{Na23} the analogue of Theorem \ref{t:reduc2}, without the periodic orbit hypothesis, namely: any diffeomorphism with vanishing asymptotic variation is almost reducible (in the natural $C^{1+bv}/C^{1+ac}$ topology). However, his proof is specific to this regularity, and ours is specific to the $C^{\infty}$ setting. To our knowledge, it is unknown whether this statement holds in intermediate regularities, starting from $C^2$. As for distortion, we only know from \cite{EBN} that among $C^2$ diffeomorphisms with vanishing asymptotic variation and with a periodic orbit, those that are distorted form a dense subset.

The above tools allowed Navas to exhibit in \cite{Na21,Na23} $C^2$ diffeomorphisms which are $C^1$ but not $C^2$-distorted (resp. almost reducible). 

At that point, if it weren't for the Main Theorem above, one could legitimately look for new length functions on groups of diffeomorphisms of higher regularity, yielding further obstructions to almost reducibility and distortion in these groups. In class $C^3$, this is closely related to the discussion on the Liouville cocycle at the end of \cite{DN}. For a $C^3$ diffeomorphism $f$ of the interval $[0,1]$, the following formula defines an $L^1$ function $c(f)$ on $[0,1]^2$ (see \cite{Na06}):
$$c(f)(x,y) = \frac{Df(x) Df(y)}{(f(x)-f(y))^2} - \frac{1}{(x-y)^2}.$$
The map $c$ satisfies a nice chain rule which implies the triangle inequality $\|c(fg)\|_{L^1}\le\|c(f)\|_{L^1}+\|c(g)\|_{L^1}$, making $L:f\mapsto \|c(f)\|_{L^1}$ a length function on $\mathrm{Diff}^3_+([0,1])$, from which one can define, as before, $L_\infty:f\mapsto \lim_{n \to \infty}L(f^n)/n$. Dinamarca and Navas asked for a characterization of those $f$ for which $L_{\infty}(f)$ vanishes. As a consequence of Theorem \ref{t:reduc2} above, we obtain that, if a $C^{\infty}$ diffeomorphism $f$ of the interval $[0,1]$ has vanishing asymptotic variation, then $L_{\infty} (f)$ also vanishes.
%\marginpar{\tiny  {\color{blue}Il semblerait que, comme le corollaire est l'intérieur d'une remarque, Corollaire n'apparaissait pas en gras, ce qui ne me plait pas trop. Par ailleurs, je pense que je préfère que l'on évite d'autres énoncés formels dans l'introduction, comme ça les mettrai quelque part sur le même plan que les théorèmes de l'introduction. J'ai donc enlevé l'énoncé formel.}}

Indeed, just like for the asymptotic variation, using the triangle inequality, one easily sees that $L_\infty$ is a $C^3$-conjugacy invariant, and that $0\le L_\infty(g)\le L(g)$ for every $g\in \Diff^3_+([0,1])$, so in fact:
$$\forall f\in \Diff^3_+([0,1]),\quad L_\infty(f)\le \inf_{h\in \Diff^3_+([0,1])}L(hfh^{-1}).$$
Furthermore, Lemma 1.2 from \cite{Na06} shows that $L$ is continuous for the $C^3$ topology, and thus for the $C^\infty$ one. Since $L(\id)=0$, it follows directly that if $f$ is almost reducible, $L_\infty(f)=0$, and Theorem \ref{t:reduc2} allows to conclude. 
Two questions remain: is the corollary above still true for $C^3$ diffeomorphisms with vanishing asymptotic variation? Does the converse hold for diffeomorphisms without hyperbolic fixed points? This last hypothesis is necessary, for Möbius transformations with hyperbolic fixed points, seen as diffeomorphisms of $[0,1]$, have vanishing $L_\infty$. \end{rem}

This article raises some natural questions.
\begin{enumerate}
\item Which elements are distorted in the group of germs at $0$ of smooth diffeomorphisms of $[0,1]$ ? Answering this question would allow to describe the distorted elements in the group $\Diff^{\infty}_+([0,1])$. Of course, the same question holds for other regularities.
\item For $1 \leq r < \infty$, are the groups $\Diff^r_+(\Ss^1)$ and $\Diff^r_c(\R)$ locally perfect ? Notice that, in the case $r=2$, it is still unknown whether those groups are perfect.
\item For $2 \leq r < \infty$, is it true that any element of $\Diff^r_+(\Ss^1)$, $\Diff^r_+([0,1])$ or $\Diff^r_c(\R)$ with vanishing asymptotic variation is almost reducible ?
\end{enumerate}

\section{Preliminaries on interval diffeomorphisms}
\label{s:classical}
%\marginpar{{\tiny \textcolor{red}{Il y a du nouveau ici (préambule et 1.1)}}}

In this section, we recall some classical results regarding diffeomorphisms of the interval, that we will use in Sections~\ref{s:reduc} and \ref{ss:perfect}. We refer the reader to \cite[chap. 4]{Na11} and \cite[chap. IV-V and Appendice 3]{Yoc} for accessible proofs of these. We only mention the infinitely smooth setting, though there are more general statements that encompass finite regularity. Of course, what we state for the interval $[0,1]$ extends to any compact interval of $\R$. 

\subsection{Generating vector fields, Mather invariant}
\label{ss:gen-vec}

%\marginpar{\tiny {\color{blue}Je te laisse ajouter la référence à Sergereart pour que tu t'habitues au bibtex (cf guide en tête d'article). Je peux aussi l'ajouter si tu préfères que je m'occupe de toute la biblio.}} 
\paragraph{Generating vector field.} Given a smooth diffeomorphism $f$ of the interval $[0,1)$ fixing only~$0$, works of Szekeres and Kopell \cite{Sz58,Ko} imply that there exists a unique $C^1$ vector field on $[0,1)$ having $f$ as its time-$1$ map.
%, and that the $C^1$ centralizer of $f$ reduces to the flow maps of this vector field. 
We will denote it by $X_f$ and call it the \emph{generating vector field of $f$}. According to Sergeraert \cite{Se}, $X_f$ is actually $C^\infty$ on $(0,1)$ but not necessarily on $[0,1)$. However, as a consequence of a result by Takens (see \cite{Ta}), if $f$ is not infinitely tangent to the identity at $0$, the vector field $X_f$ is smooth on $[0,1)$. Hence, the problem of regularity of $X_f$ at $0$ can occur only when the diffeomorphism $f$ is infinitely tangent to the identity at $0$.
%{\color{green} unless $f$ is not infinitely tangent to the identity at $0$ (this last claim is due to Takens \cite{Ta}).} 
Throughout the article, we will abbreviate ``infinitely tangent to the identity'' by ``ITI''.  

%{\color{blue} Comme mon cerveau a des difficultés avec les double négations, je te propose de remplacer le vert par les deux nouvelles phrases suivantes.
%
%} 

It also follows from Kopell's celebrated ``Lemma'' that, if $f$ and $g$ are as above and $\f$ is a $C^1$-diffeomorphism of $[0,1)$ conjugating $f$ to $g$, then $\f_*X_f = X_g$. 

\paragraph{Mather invariant.} If $f$ now denotes a smooth diffeomorphism of $[0,1]$ without interior fixed point, $f$ has two generating vector fields, one on $[0,1)$ (still denoted $X_f$) and one on $(0,1]$ (that we will denote by $Y_f$). Following Yoccoz \cite{Yoc}, we will denote by $(f_t)_{t\in\R}$ and $(f^t)_{t\in\R}$ their respective flows. Generically, these two vector fields do not coincide, and the Mather invariant of $f$, that we introduce now, ``measures'' this lack of coincidence.

Given $p,q\in(0,1)$, the maps $\psi_{X,p}:t\mapsto f_t(p)$ and $\psi_{Y,q}:t\mapsto f^t(q)$ define smooth diffeomorphisms from $\R$ to $(0,1)$ which both conjugate the translation $t\mapsto t+1$ on $\R$ to $f_{|(0,1)}$. Hence $\psi_{Y,q}^{-1}\psi_{X,p}=:M_f^{p,q}$ is a smooth diffeomorphism of $\R$ commuting with the translation by $1$, and one easily sees that changing the base points $p$ and $q$ boils down to pre- and post-composing $M_f^{p,q}$ by translations. The class of this diffeomorphism up to these pre- and post-compositions is the so-called \emph{Mather invariant of $f$}, which we will denote by $M_f$.

We will say that $f$ has a trivial Mather invariant if $M_f$ is the class of translations, which corresponds to the case where $X_f$ and $Y_f$ coincide, \emph{i.e.} where $f$ embeds in a $C^1$ flow on the whole interval $[0,1]$. 

%The Mather invariant is 

%\textcolor{red}{J'ai l'impression qu'il n'y a rien à rajouter ici finalement.}\medskip

\subsection{Sufficient conditions of conjugacy on the interval}
\label{s:crit}

%On suppose que generating vector fields, Mather invariant, ont été introduits ailleurs.\medskip

%\textcolor{red}{Maybe this should go to the previous preliminary section}\medskip

%In this paragraph, $a$ and $b$ denote real numbers satisfying $a<b$ and $r$ an integer greater than or equal to $2$. 

%\textcolor{red}{Se placer sur $[0,1]$ ou $[a,b]$ ?}\medskip
%\marginpar{\tiny {\color{blue}Je pense que l'on peut garder cette notation, ne serait-ce que parce que ça allège les notations dans la définition et le lemme suivants.}}
In this section, we provide a conjugacy criterion (Lemma \ref{l:crit} below) that will be used in Sections~\ref{s:reduc} and \ref{ss:perfect}, for elements of the set $\Ddz([0,1])$ of smooth diffeomorphisms of the segment without interior fixed point and with trivial Mather Invariant (equivalently: diffeomorphisms which embed in the flow of a $C^1$ vector field without interior fixed point).
% \textcolor{red}{(pas clair que la notation $\Ddz$ soit très utile. On ne s'en sert que 2 fois finalement)}. 
Given such a diffeomorphism, we will denote by $X_f$ its unique generating vector field and by $(f^t)_{t\in\R}$ the flow of $X_f$. In order to state the conjugacy criterion, it will be handy to dispose of the following notation:

\begin{defn}
\label{d:tf}
Given $f\in\Ddz([0,1])$, for every $p,q\in(0,1)$, there exists a unique $\tau\in\R$ such that $f^\tau(p)=q$. We will denote this number by $\tau_f(p,q)$, which is none but the time necessary to go from $p$ to $q$ following the flow of $X_f$ (this ``time'' can be negative). 
\end{defn}

Note that $\tau_f(p,q)$ is given by the simple formula:

\begin{equation}
\label{e:tf}
\tau_f(p,q) = \int_p^q \frac1{X_f}.
\end{equation}

Lemma \ref{l:crit} below is a detailed version, in some restricted situation, of the following classical fact. 

\begin{prop}[Mather \cite{Mather}, see also \cite{Yoc}]
\label{p:Mather}
Two $C^2$ diffeomorphisms of the interval without interior fixed point are $C^1$-conjugated if and only if their germs at the endpoints are $C^1$-conjugated and they have the same Mather invariant.
\end{prop}
% (cf. Remark \ref{r:mather}). 
This provides a complete $C^1$-conjugacy invariant for $C^2$ diffeomorphisms without interior fixed point (beware that this does not generalize to conjugacy in higher regularity). Below, the condition on the Mather invariant is trivially ensured since we only consider diffeomorphisms with trivial Mather invariant,  as in the rest of the article. 
%ensured by the vanishing of the asymptotic variation of both diffeomorphisms (which thus both have a trivial Mather invariant according to \cite[Corollary 2]{EN21}). 

%We postpone the (classical) proof of this statement and its corollary \ref{c:crit} to the end of the section.

\begin{lem}[Conjugacy criterion]
\label{l:crit}
%\marginpar{\tiny  \textcolor{red}{faudrait-il l'écrire pour n'importe quel invariant de Mather même si on ne l'utilise que pour un trivial ici?} {\color{blue} Je ne pense pas que ce soit la peine. On n'en a pas besoin et il me semble que la démo est un peu plus concrète dans ce cas facile}}
Let $f,g\in\Ddz([0,1])$ and assume that there exist $C^1$ diffeomorphisms $\f_0$ and $\f_1$ between neighborhoods of $0$ and $1$ respectively such that $\f_0 f \f_0^{-1}=g$ near $0$ and $\f_1 f \f_1^{-1}=g$ near $1$.

Then there is a unique $C^1$ diffeomorphism $\f$ of the whole interval $[0,1]$ coinciding with $\f_0$ near $0$ and which conjugates $f$ to $g$. Moreover, $\f$ is $C^\infty$ on $(0,1)$ and, near $1$, it coincides with~$\f_1\circ f^\sigma$ for some real number $\sigma$ (which depends on $f$, $g$, $\f_0$ and $\f_1$). 

More precisely, for any $x_0,x_1$ close enough to $0$ and $1$ respectively, one has:
$$\sigma = \tau_f\bigl(x_0,x_1\bigr)-\tau_g\bigl(\f_0(x_0),\f_1(x_1)\bigr)
%\quad\text{\textcolor{red}{(Ou l'opposé)}}
.$$
In particular, if this quantity is an integer, or if $X_f$ is $C^\infty$ at $1$, the regularity of $\f$ at $1$ is that of $\f_1$.
%
%More precisely, if for some $p\in U_a$ and some $\tau$ and $\tau'\in\R$ one has $g^{\tau'}(\f_a(p))\in V_b$ and $\f_b^{-1}(g^{\tau'}(\f_a(p)))=f^\tau(p)$, then $\sigma=\tau-\tau'$ \textcolor{red}{ou l'inverse}.
\end{lem}

This lemma implies that:
\begin{itemize}
\item if the germs of $f$ and $g$ at each endpoint are smoothly conjugated and the $C^1$ flow in which $f$ embeds is smooth at at least one endpoint (which, according to Takens \cite{Ta}, happens as soon as $f$ is not ITI at this endpoint), then $f$ and $g$ are smoothly conjugated; 
\item the conclusion of the previous item also holds without assumption on the regularity of the flow of $f$ at the endpoints but with the ``synchronicity'' assumption that, with the notations of the lemma, $\sigma$ is an integer;
\item if what we want is construct a nice diffeomorphism $g$ conjugated to some given $f$ while prescribing the germs of the conjugacy at the endpoint, we need to be able to control this quantity $\sigma$ through the construction.
%\textcolor{red}{(redites)} 
\end{itemize}

%\textcolor{red}{
%Deux difféos comme ci-dessus sont $C^1$-conjugués ssi leurs germes au bord sont conjugués (théorie de Mather), et si les bords sont non ITI, ils sont alors aussi $C^\infty$-conjugués. Ceci n'est pas nécessairement vrai si un bord est ITI, et quand on passe au cercle, ça ne marche plus dans le premier cas, à cause du shift qu'il faut introduire (et qui par contre ne pose pas de pb s'il y a un point de fixe ITI. Quel casse-tête ! Ça va être chaud à organiser !)
%}

\begin{proof}[Proof of Lemma \ref{l:crit}]
%\marginpar{{\tiny \textcolor{red}{voilà c'est corrigé}}}
%One can define a $C^1$ diffeomorphism $\psi$ of $[0,1]$ coinciding with $\f_0$ near $0$ and $\f_1$ near $1$. Conjugating $f$ by $\psi$, one reduces to the case where the two (new) diffeomorphisms $f$ and $g$ coincide near the endpoints (\emph{i.e.} to the case where $\f_0$ and $\f_1$ are the identity near $0$ and $1$ respectively). \textcolor{red}{Problème de rédaction. $f$ n'est plus $C^\infty$...} 
Without loss of generality, we can assume $f>\id$ and $g>\id$ on $(0,1)$.

We define $\f$ step by step. On some neighborhood $[0,x]$ of $0$, we must let $\f=\f_0$, and we have $\f f =g\f$ on $[0,f^{-1}(x)]$ provided $x$ is small enough. To have this relation on $[f^{-1}(x),x]$, we must let $\f=g \f_0 f^{-1}$ on $[x,f(x)]$. Then to have it on $[x,f(x)]$, we must let $\f=g^2 \f_0 f^{-2}$ on $[f(x),f^2(x)]$, and so on. We get some (completely determined) $\f$ defined on $[0,1)$. By construction, $\f=g^n\f f^{-n}$ on $(0,1)$ for every $n\in\Z$. Now for any $x\in(0,1)$, there exists $n\in\N$ such that $f^{-n}(x)$ belongs to a neighborhood of $0$ where $\f=\f_0$ is $C^1$, so $\f$ is $C^1$ on $[0,1)$. 
%It has been defined piecewise, but it is $C^\infty$ on $(0,1)$. Indeed, by construction, $\f f = g\f$ on the whole of $[0,1)$, so $\f = g^n\f f^{-n}$ for every $n$, and so for every $y$ in $(0,1)$,  if $n$ is such that $f^{-n}(y)$ is close to $0$ (in the \end{proof}

In particular, $\f$ must 
%conjugate the $C^1$ centralizer of $f_{|[0,1)}$ to that of $g_{|[0,1)}$, but those are made of the flow-maps of the (left) generating vector fields $X_f$ and $X_g$. So $\f$ must 
send $X_f$ to $X_g$ on $[0,1)$ (cf. Section \ref{ss:gen-vec}), and in particular on $(0,1)$, where we thus have $D\f=\frac{X_g\circ\f}{X_f}$. This gives the $C^\infty$ regularity of $\f$ on $(0,1)$ by induction, since, according to Sergeraert, $X_f$ and $X_g$ are $C^\infty$ on $(0,1)$. 

Now this is where we use the triviality of the Mather invariants: since $X_f$ and $X_g$ are also the right generating vector fields of $f$ and $g$, and $\f_1$ conjugates $f$ to $g$ near $1$, one also has $(\f_1)_*X_f=X_g$ near $1$. Hence, on $(y,1)$ for $y$ close o $1$, $(\f_1^{-1}\f)_*X_f=X_f$, which implies that $\f_1^{-1}\f$ coincides there with a flow-map $f^\sigma$ of $X_f$, and this of course extends by continuity at $1$. In particular, $\f$, as $\f_1$ and $f^\sigma$, is $C^1$ near $1$. \medskip

To prove the last part of the statement, observe that, for $x_0$ and $x_1$ close enough to $0$ and $1$ respectively, using two changes of variables (for the second and fifth equality),
\begin{align*}
\tau_f(x_0,x_1) &= \int_{x_0}^{x_1}\frac1{X_f} =  \int_{\f(x_0)}^{\f(x_1)}\frac1{\f_*X_f}
=\int_{\f_0(x_0)}^{(\f_1\circ f^\sigma)(x_1)}\frac1{X_g}\\&\\
&=\tau_g\bigl(\f_0(x_0),\f_1(x_1)\bigr)+\int_{\f_1(x_1)}^{\f_1(f^\sigma(x_1))}\frac1{X_g}\\
&=\tau_g\bigl(\f_0(x_0),\f_1(x_1)\bigr)+\int_{x_1}^{f^\sigma(x_1)}\frac1{X_f} = \tau_g\bigl(\f_0(x_0),\f_1(x_1)\bigr)+\sigma,
\end{align*}
which concludes the proof.
\end{proof}

\section{Almost reducibility of diffeomorphisms}
\label{s:reduc}

\subsection{Main ideas and central technical statements}

%\marginpar{\tiny Il y avait trop de parenthèses à mon goût. Perso, je réserve les parenthèses aux citations internes et externes. Si on veut mentionner quelquechose de vraiment annexe, je préfère les notes de bas de page qui n'interrompent pas la lecture du texte.}
This section is devoted to the proof of Theorem \ref{t:reduc2}. This theorem deals with three different groups: $\Diff^\infty_+(\Ss^1)$, $\Diff^\infty_+([0,1])$ and $\Diff^\infty_c(\R)$. Let us first observe that the latter easily reduces to the case of the segment. Indeed, let $f\in\Diff^\infty_c(\R)$ be a diffeomorphism with vanishing asymptotic variation, that we are trying to conjugate closer and closer to the identity. Without loss of generality, we can assume $f$ is supported in $[0,1]$, and $f$ is of course ITI at $0$ and $1$. Hence if we admit Theorem \ref{t:reduc2} in the case of the segment, the restriction of $f$ to $[0,1]$ can be conjugated closer and closer to the identity by smooth diffeomorphisms of $[0,1]$. These can easily be extended to smooth diffeomorphisms of $\R$ with compact support, say in $[-1,2]$, and the conjugates of $f$ by these extensions are the extensions of the conjugates on $[0,1]$ by the identity, so they too are  closer and closer to the identity on $\R$, as required. \medskip

The case of the circle is a little bit more subtle. Let us start with a diffeomorphism $f$ of the circle with at least one fixed point, which can be assumed to be $0\in\R/\Z$, and with vanishing asymptotic variation. Moreover, if $f$ has ITI fixed points, we assume $0$ is one of them. Let $\hat f$ be the diffeomorphism of $[0,1]$ obtained by ``opening the circle'' at $0$, and assume it can be brought closer and closer to the identity by conjugacy on $[0,1]$. If the one-sided right $\infty$-jets of the conjugacies at~$0$ and the left one at $1$ match, we can ``close back the circle'' and get smooth diffeomorphisms, both for the conjugacies and the conjugates. The following statement claims that this can be ensured, and thus implies Theorem \ref{t:reduc2} in the case of a diffeomorphism with at least one fixed point.

\begin{thmintro}
\label{t:reduc3}
Any $f\in\Diff^\infty_+([0,1])$ with vanishing asymptotic variation can be conjugated arbitrarily close to the identity. Furthermore, if $f$ is ITI at both endpoints, the conjugacies can be required to be ITI there as well, while if $f$ is nowhere ITI, the conjugacies can be required to coincide with homotheties of the same factor near the endpoints. \end{thmintro}

Theorem \ref{t:reduc2} in the case of a diffeomorphism of the circle with a nonfixed periodic orbit is deduced from Theorem \ref{t:reduc3} in Subsection \ref{ss:rational}.

We are thus reduced to proving Theorem \ref{t:reduc3}. Let us present the main ideas and difficulties, starting from the simplest situation and adding one complication at a time. Let $f$ be a smooth diffeomorphism of the interval with vanishing asymptotic variation or, equivalently, which embeds in a $C^1$ flow and has only parabolic fixed points.  Fix $r\in\N^*$.\medskip

\textbf{We first assume $f$ has no interior fixed point \textbf{(Restriction 1)}. }
\begin{itemize}
\item \textbf{(Idea 1)} The first observation is that it is not difficult to conjugate $f$ to make it close to the identity near the boundary. Indeed, conjugating $f$ by homotheties of big ratios near the endpoints, we get parabolic germs of diffeomorphisms whose higher order derivatives at $0$ and $1$ are as small as we like. We could then interpolate between these germs to get a diffeomorphism $\tilde f$ of the whole interval which is $C^r$-close to the identity. However \textbf{(Problem 1)}, for this diffeomorphism to be \emph{globally} even just $C^1$-conjugated to $f$, it needs to embed in a $C^1$ flow itself (cf. Proposition \ref{p:Mather}).
\item \textbf{(Idea 2)} It is thus more natural to work on the generating vector field of $f$ rather than $f$ itself, \emph{i.e.} to conjugate, or rather push-forward, the generating vector field by homotheties of big ratio near the boundary, interpolate between the new germs of vector fields and take $\tilde f$ to be the time-$1$ map of the resulting vector field. This raises several problems. First, the germs of generating vector fields are generally only $C^1$ at the boundary \textbf{(Problem 2)}, so we cannot hope to get a $C^r$ control on the time-$1$ map of the resulting vector field, or at least not that easily.
\item \textbf{Assume for a while the generating vector field of $f$ is $C^\infty$ at the endpoints \textbf{(Restriction 2)}}. Then the above idea indeed works: one can construct a smooth vector field $\tilde X$ which is as $C^r$-close to the zero vector field as we like, and thus whose time-$1$ map $\tilde f$ is close to the identity, and which is smoothly conjugated to the initial one near the boundary points by homotheties of the same ratio if we like. By the conjugacy criterion of the previous section (Lemma \ref{l:crit}), its time-$1$ map $\tilde f$ is conjugated to $f$ by a $C^1$ diffeomorphism which is smooth in the interior, coincides with a homothety near $0$ and with a homothety \emph{composed with a flow map $\tilde f^\sigma$ of $\tilde X$} near $1$, and is thus smooth on the whole interval. This yields Theorem \ref{t:reduc2} in the particular case of the segment $[0,1]$, without interior fixed point and with smooth generating vector field. But recall that, in order to deal with the circle case, we need the stronger statement of Theorem \ref{t:reduc3}. Namely, we would like the flow map $\tilde f^\sigma$ above to be the identity near $1$, \emph{i.e.} we would like $\sigma$ to be $0$ \textbf{(Problem 3)}, which requires some extra care in the interpolation. 
\end{itemize}
Proposition \ref{l:woITI} below claims that all of this can indeed be done and also allows to deal with a finite number of interior fixed points, while keeping the restriction on the smoothness of $X$ (cf. Section \ref{ss:easy-case}).\medskip

We now need to get rid of the above restrictions, \emph{i.e.}, roughly, to consider the case where there are possibly infinitely many ITI fixed points. If there aren't, the fixed points are isolated and thus in finite number, and the generating vector field is smooth according to Takens \cite{Ta}. \medskip

Let us first try to get rid of Restriction 2, while keeping Restriction 1. In other words, \textbf{let us start with a smooth diffeomorphism $f$ without interior fixed points whose generating vector field $X$ is not smooth at the boundary}. To fix ideas, let us consider the case where $X$ is smooth neither at $0$ nor at $1$, which implies that $f$ is ITI at both endpoints.

\begin{itemize}
\item In order to find an $\tilde f$ smoothly conjugated to $f$ and close to the identity, one can first keep $X$, and thus its time-$1$ map, unchanged near $0$ and $1$, where $f$ is already close to $\id$, and try to interpolate between the two germs of vector fields \emph{in such a way that the time-$1$ map of the resulting vector field $\tilde X$ be close to the identity}. This time, this cannot be guaranteed by making sure $\tilde X$ is $C^r$-small \emph{on $[0,1]$}, for $\tilde X$ is not even $C^2$ at $0$ and $1$ \textbf{(Problem 4)}. But in fact, roughly, we only need it to be $C^r$-small \emph{outside a neighborhood of the boundary where it coincides with~$X$}, for in this neighborhood, its time-$1$ map is, as $f$, close to $\id$. The key fact here is that, according to Lemma \ref{lem:regularitypoints}, though $X$ is not smooth at $0$, say, there are regions arbitrarily close to $0$ where all the derivatives of $X$ up to order $r$ are small. The idea is then to keep $X$ unchanged between $0$ and such a region, interpolate between $X$ and a constant vector field in this region, do the same near $1$ and then interpolate between the two constant pieces of vector fields. The first interpolation near each of the endpoints is pretty much what is done in Proposition \ref{p:interp-ITI} below, in the more general context of a possible accumulation of non-ITI fixed points at $0$. 
\item According to Lemma \ref{l:crit}, after the second interpolation, between the left and the right hand sides of the interval, the time-$1$ map of the resulting vector field is conjugated to $f$ by a $C^1$ diffeomorphism $\f$ which is smooth in the interior and coincides with the identity near~$0$ and with some flow map of $\tilde X$ near $1$, or of $X$, since $X$ and $\tilde X$ coincide there. Hence $\f$ is smooth on $[0,1)$ but not necessarily at $1$ since $X$ is not smooth there \textbf{(Problem 5)}. We can guarantee the smoothness of the conjugacy at $1$ by ensuring that the corresponding time of the flow is an integer, and thus the conjugacy coincides with a power of $f$ near $1$, which is smooth. This is indeed possible, and not too hard in the restricted case we are currently dealing with, and the resulting statement is roughly that of Proposition~\ref{p:reduc} below in the case without interior fixed point. To obtain the same statement with interior non-ITI fixed point, one must work a bit harder and throw in Lemma~\ref{l:woITI}.
%\item In order to find an $\tilde f$ smoothly conjugated to $f$ and close to the identity, the idea is to keep $X$ (and thus its time-$1$ map) unchanged near $0$ (where $f$ is already close to $\id$), to conjugate it by a homothety near $1$ if necessary (as in the previous restricted case) and to interpolate between the two germs \emph{in such a way that the time-$1$ map of the resulting vector field $\tilde X$ be close to the identity}. This time, this cannot be guaranteed by making sure $\tilde X$ is $C^r$-small \emph{on $[0,1]$}, for $\tilde X$ might not even be $C^2$ at $0$. But in fact (roughly) we only need it to be $C^r$-small \emph{outside a neighborhood of $0$ where it coincides with $X$} (for in this neighborhood of $0$, its time-$1$ map is, as $f$, close to $\id$). This situation motivates Proposition \ref{p:interp-ITI} below which, combined to Lemma \ref{l:woITI}, actually allows to settle the case \emph{without ITI fixed point} (cf. Proposition \ref{p:reduc}, proved in Section \ref{ss:reduc}). 
\item The completely general case easily reduces to the one of Proposition \ref{p:reduc} (cf. Section \ref{ss:general-case}). 
%\item Note that again, once we have constructed a $C^1$ vector field $\tilde X$ smoothly conjugated to $X$ near the boundary and whose time-$1$ map $\tilde f$ is close to the identity, we have not won yet. Indeed, $\tilde f$ is conjugated to $f$ by a diffeomorphism which coincides near $0$ and $1$ with a time 
%\textcolor{red}{reprendre ici}
\end{itemize}
%More precisely, what we show in this context is Proposition \ref{p:interp-ITI} below which, combined to Lemma \ref{l:woITI}, actually allows to settle the more general case \emph{without interior ITI fixed point} (cf. Proposition \ref{p:reduc}, proved in Section \ref{ss:reduc}). 
%

%Roughly 2 difficulties: 
%\begin{itemize}
%\item prove Theorem \ref{t:reduc3} for a diffeo of the interval without interior fixed point (and inside this case, it is much easier if the generating vector field is smooth);
%\item for a general diffeo, make sure the conjugacies built on each component glue up nicely (not too difficult if finite number of fixed point, though subtlety when ``closing the circle'', less straightforward if accumulation of fixed point).
%\end{itemize}
%In order to spread the difficulty, we are first going to deal with the case of a smooth generating vector field with finitely many fixed points, and then with the general case (with at least one ITI fixed point otherwise by Takens bla bla bla). But before that, need criterion of conjugacy. 
%

We now state the three results alluded to above, which are the heart of the proof of Theorem~\ref{t:reduc3}. In Section \ref{ss:easy-case}, we show that, for diffeomorphisms which embed in a smooth flow with finitely many fixed points, the statement of Theorem \ref{t:reduc3} directly follows from Proposition \ref{l:woITI}. In Section \ref{ss:general-case}, we prove that, in the general case, Theorem \ref{t:reduc3} easily reduces to Proposition \ref{p:reduc}. Then we devote one section to each of the proofs of Propositions \ref{l:woITI}, \ref{p:interp-ITI} and \ref{p:reduc}.

%\marginpar{\tiny \textcolor{red}{lui est sur $[a,b]$ et les autres sur $[0,1]$. Homogénéiser} \color{blue}{C'est fait.}}
\begin{prop} 
\label{l:woITI}
Let $\eta>0$, $r\in\N$ and $X$ be a $C^\infty$ vector field 
%defined on a neighborhood of 
on $[0,1]$ vanishing only at $0$ and~$1$ and with vanishing derivative at these points. Then there exists $\lambda_X>1$ such that, for any $\lambda\ge\lambda_X$, there exists $\f\in\Diff^\infty_+([0,1])$ such that:
\begin{enumerate}
\item $\|\varphi_*X\|_r<\eta$;
\item on some neighborhood of each endpoint,
% $c\in\{0,1\}$, 
$\f$ coincides with the homothety of ratio $\lambda$ centered at that point.
\end{enumerate}
\end{prop}

\begin{prop}
\label{p:interp-ITI}
Let $X$ be a $C^1$ vector field on $[0,1)$, $C^\infty$ and nowhere infinitely flat on $(0,1)$, whose time-$1$ map $f$ is $C^\infty$ on $[0,1)$ and ITI at $0$. 
%For any non-fixed point $x$, let $x^+=\min\{y\in\Fix(f); y>x\}$. 
% (resp. $x^-=\max\{y\in\Fix(f); y<x\}$). 
Let $\eta >0$ and $r\in\N^*$. Then there exists a non-fixed point $0<x_0<1-\eta$ of $X$ arbitrarily close to $0$ and a $C^1$ vector field $Y$ on $[0,1)$ with the following properties.
\begin{enumerate}
\item $Y_{|[0,x_0]}=X_{|[0,x_0]}$.
\item The vector field $Y$ is equal to the constant vector field $X(x_0)$ on $[x_0+\eta,1)$.
\item The vector field $Y$ is $C^\infty$ on $(0,1)$,
\item $\left\| Y_{|[f^{\mp1}(x_0),x_0+\eta]} \right\|_r \leq \eta$, where $f^{\mp1}(x_0)=\min(f^{-1}(x_0),f(x_0))$.
\item The vector field $Y$ does not vanish on $[x_0,x_0+\eta]$.
\end{enumerate}
\end{prop}

\begin{prop}
\label{p:reduc}
Any smooth diffeomorphism $f$ of $[0,1]$ without interior ITI fixed point and with vanishing asymptotic variation is almost reducible. More precisely, for any $r\in\N$ and $\eps>0$, there exists $\f\in\Diff_+^\infty([0,1])$ such that $d_r(\f f \f^{-1},\id)<\eps$ and such that $\f$ coincides with a power of $f$ at the possible ITI boundary points.
\end{prop}

%\textcolor{red}{(shouldn't the main protagonist of this lemma be $X$ and not $f$? Though we need to mention $f$ because it is important that the time-$1$ map of $X$ be $C^\infty$ at $0$).}\medskip

%Let us now prove Proposition \ref{p:reduc}.

We will also need the following classical result:

\begin{lem}
\label{l:flow}
For any $\eps>0$ and any $r\in\N$, there exists $\eta>0$ such that for any $C^r$ vector field~$Y$ on $[0,1]$ vanishing at $0$ and $1$, if $\|Y\|_r<\eta$, its time-$1$ map $g$ satisfies $\|g-\id\|_r<\eps$.
\end{lem}

\subsection{Proof of Theorem \ref{t:reduc3} in the case of a smooth generating vector field with finitely many zeroes}
\label{ss:easy-case}
%\textcolor{red}{Il vaudrait mieux l'écrire pour les champs vu l'usage fait ensuite.}

%\begin{lem}
%\label{l:flow}
%For any $\eps>0$ and any $r\in\N$, there exists $\eta>0$ such that for any $C^r$ vector field~$Y$ on $[0,1]$ vanishing at $0$ and $1$, if $\|Y\|_r<\eta$, its time-$1$ map $g$ satisfies $\|g-\id\|_r<\eps$.
%\end{lem}

%\textcolor{red}{Il vaut sans-doute mieux dire dans l'énoncé ce qu'on fait concrètement et dire que c'est conjugué par ce qu'on veut grâce aux critères de conjugaison expliqués antérieurement...). }

%\textcolor{red}{\`A nettoyer suite à restructuration}\medskip

Here, we restrict to the case of a smooth diffeomorphism $f$ of the interval $[0,1]$ arising as the time-$1$ map of a smooth vector field $X$ with finitely many zeroes $0=a_0<\dots<a_n=1$. Fix $r\in\N$ and $\eps>0$. For every $i\in[\![0,n-1]\!]$, let $X_i$ denote the restriction of $X$ to $I_i=[a_i,a_{i+1}]$ (and $f_i$ its time-$1$ map, which is the diffeomorphism of $I_i$ induced by $f$). Let us apply Proposition~\ref{l:woITI} to each of these vector fields, with an $\eta>0$ associated to $\eps$ and $r$ as in Lemma \ref{l:flow}, and denote $\lambda:=\max_{0\le i\le n-1}\lambda_{X_i}$. According to Proposition \ref{l:woITI}, for every $i$, there exists a smooth diffeomorphism $\f_i$ of $I_i$ coinciding with a homothety of ratio $\lambda$ near $a_i$ and $a_{i+1}$ and such that $\|(\f_i)_*X_i\|_r<\eta$, which implies $g_i:=\f_i f_i\f_i^{-1}$ is $\eps$-$C^r$-close to the identity by the choice of $\eta$. Define $g$ as the homeomorphism of $[0,1]$ coinciding with $g_i$ on $I_i$ for every $i$. By construction, $g_i$ is smooth on the interiors of the $I_i$'s, and also at the $a_i$'s since there, it is simply conjugated to $f$ by a homothety. Moreover, it is $\eps$-$C^r$-close to the identity. Finally, the $\f_i$'s clearly match up as a smooth diffeomorphism $\f$ of the interval conjugating $f$ to $g$ and coinciding with a homothety of the same ratio at the endpoints, which concludes the proof of Theorem \ref{t:reduc3} in this nice simple case.

\subsection{Reduction of Theorem \ref{t:reduc3} to Proposition \ref{p:reduc} in the general case}
\label{ss:general-case}
%\subsubsection{Reduction to the case without interior ITI fixed points}
%\label{ss:wointITI}

%Il this subsection, we prove that Theorem \ref{t:reduc3} reduces to Proposition \ref{p:reduc}.
%Let us first deal with the case of the interval $[0,1]$, and the other two will follow. So 
Let $f\in\Diff^\infty_+([0,1])$ and fix $\eps>0$ and $r\in\N$. The case where $f$ is nowhere ITI follows directly from Section \ref{ss:easy-case} above, so we assume from now on that $f$ has at least one ITI fixed point. 

\begin{claim}
\label{c:subdiv}
There is a subdivision $0=x_0<\dots<x_n=1$ such that each $x_i$ which is not an endpoint is ITI and for every $i\in[\![0,n-1]\!]$, either $d_r(f_{|[x_i,x_{i+1}]},\id)<\eps$ (interval of type A) or $f$ has no ITI fixed point inside $(x_i,x_{i+1})$ (interval of type B). 
\end{claim}

\begin{proof}[Proof of the claim] Consider the maximal disjoint intervals whose boundary points different from~$0$ and $1$ are ITI fixed points and on which the restriction of $f$ is $\eps$-$C^r$-close to the identity. Near an accumulation point of the set of ITI fixed points of $f$, $f$ is (arbitrarily) close to the identity. This implies on the one hand that the former intervals are necessarily in finite number, and that the complement of their union contains at most finitely many ITI fixed points, which easily gives the desired subdivision.
\end{proof}

Let $I_0,\dots,I_{n-1}$ denote the intervals of the subdivision given by Claim \ref{c:subdiv}. We want to define a smooth diffeomorphism $g$ conjugated to $f$ and $\eps$-$C^r$-close to the identity, and we require, in addition, that the conjugacy coincides with a power of $f$ at the possible boundary ITI points. Let us proceed by induction. 
First, if $I_0$ is of type A, we let $g=f$ on $I_0$ (and they are conjugated by the identity, which is a power of $f$ near each boundary point). Otherwise, we apply Proposition~\ref{p:reduc} and again, this gives a conjugacy which is a power of $f$ at each ITI endpoint of $I_0$. Assume we have constructed $g$ on $I=\cup_{j=0}^k I_j=[0,x_{k+1}]$ for some $k<n-1$. Necessarily, $x_{k+1}$ is ITI so by induction, the conjugacy already obtained on $I$ is a power of $f$ near $x_{k+1}$. If $I_{k+1}$ is of type A, we let $g=f$ on $I_{k+1}$. The gluing at $x_{k+1}$ is smooth since $g$ is itself ITI on both sides of $x_{k+1}$ by construction. Furthermore, the restriction of $g$ to $I_{k+1}$ is conjugated to that of $f$ by the identity, but also by any power of $f$, which shows that the new $g$ is smoothly conjugated to $f$ on $\cup_{j=0}^{k+1} I_j$. If $I_{k+1}$ is of type B, one applies Proposition \ref{p:reduc} and again both $g$ and the conjugacies glue up nicely at $x_{k+1}$, and the conjugacy coincides with a power of $f$ near $x_{k+2}$ if $f$ is ITI there. This completes the induction. 

%The case of diffeomorphisms of $\R$ with compact support follows immediately since they are ITI at the boundary of their support, so one can apply the above to the restriction to the support and both the conjugacy and the conjugate are ITI at the boundary and can thus be extended smoothly to $\R$ by the identity.\medskip
%
%Finally, on the circle, one only needs to consider the case of a diffeomorphism with at least one ITI fixed point since the other case has been dealt with in the previous section. One can thus open the circle at this point, apply the above to the diffeomorphism thus obtained, and again, both the conjugate and the conjugacy will be ITI at the boundary so one can ``close the circle back'' smoothly.
%
%\textcolor{red}{Finir preuve de la réduction}. \medskip

%\begin{lem}
%\label{l:glue}
%Let $Y$ be a 
%\end{lem}

\subsection{Proof of Proposition \ref{l:woITI}}

%Commencer par un auto-plagiat de \cite{EN21} et ajuster.
Let $\eta, r, X$ be as in the statement, and let $f$ denote the time-$1$ map of $X$. Without loss of generality, we can assume that $X>0$ on $(0,1)$. 
%Let $\eta>0$ be as in Lemma \ref{l:flow}. 
According to Lemma \ref{l:crit} and Formula \eqref{e:tf}, it suffices to show that, for any sufficiently large $\lambda$, one can construct a vector field $Y$ with the following properties:
\begin{enumerate}
\item on $[0,\frac14]$ (resp. $[\frac34,1]$), $Y$ is the push-forward of $X$ by the homothety centered at $0$ (resp.~$1$) and of ratio $\lambda$;
%\item near $0$ (resp. $1$), $Y$ is the push-forward of $X$ by the homothety of ratio $\lambda$ and center $0$ (resp. $1$);
\item $\|Y\|_r<\eta$;
\item for some $x_0\in(0,\frac1{4\lambda})$, 
$$\int_{x_0}^{1-x_0}\frac1X = \int_{\lambda x_0}^{1-\lambda x_0}\frac1Y.$$
\end{enumerate}
Indeed, by Lemma \ref{l:crit}, the time-$1$ map of such a $Y$ will be conjugated to that of $X$ -- and thus $Y$ will be the push-forward of $X$ (cf. Section \ref{ss:gen-vec}) -- by a diffeomorphism $\f$ satisfying all the required properties. \medskip

Let us thus construct such a $Y$. For this, we first need some preparation. Let us pick a smooth map $\rho:[0,1]\to[0,1]$ which is equal to $1$ on $[0,\frac18]\cup[\frac78,1]$, to $0$ on $[\frac14,\frac34]$, and which is decreasing on $[\frac18,\frac14]$ and increasing on $[\frac34,\frac78]$, and let $A:=\|\rho\|_r>0$. Next, for every $\lambda>2$, let $I^0_\lambda = [0,\frac1{4\lambda}]$, $I^1_\lambda=[1-\frac1{4\lambda},1]$,
$J_\lambda=[\frac1{2\lambda},1-\frac1{2\lambda}]$, $K_\lambda=[\frac12-\frac1{4\lambda}, \frac12+\frac1{4\lambda}]$, and let $\f_\lambda$ be a smooth diffeomorphism of $[0,1]$ such that:
\begin{itemize}
\item for $i=0$ and $1$, $\f_\lambda$ coincides with the homothety of ratio $\lambda$ centered at $i$ on $I^i_\lambda$, 
\item $\f_\lambda$ sends $J_\lambda$ to $K_\lambda$ homothetically (and its ratio is thus $(2\lambda-2)^{-1}$),
\item $D\f_\lambda$ is decreasing on $[\frac1{4\lambda},\frac1{2\lambda}]$ and increasing on $[1-\frac1{2\lambda},1-\frac1{4\lambda}]$ (cf. Figure \ref{f:graphphi}).
\end{itemize}

\begin{figure}[ht]
\begin{center}
\includegraphics[scale=0.7]{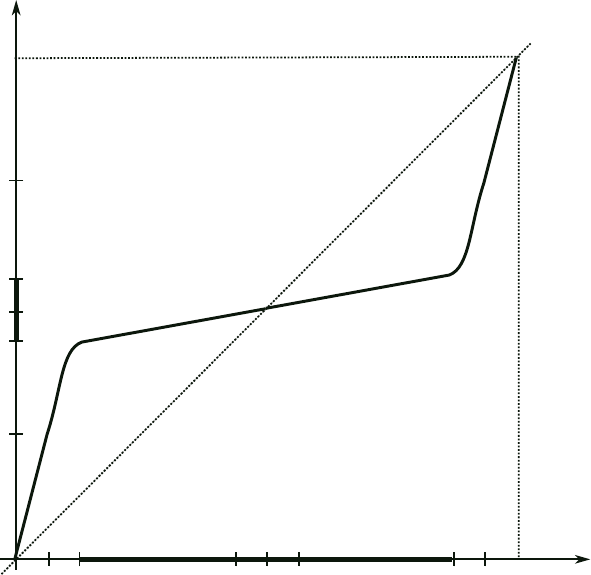}
\put(-112,-5){$\scriptstyle{J_\lambda}$}
\put(-210,88){$\scriptstyle{K_\lambda}$}
\put(-70,86){$\scriptstyle{\varphi_\lambda}$}
\caption{Graph of $\varphi_\lambda$}
\label{f:graphphi}
\end{center}
\end{figure}

%This can be done in such a way that the absolute values of the derivatives of all the $u_\alpha$ up to the order $r$ are bounded above by some $A>0$ independent of $\alpha$. 

For every $\lambda>0$, let $X_\lambda = (\f_\lambda)_*X$. This vector field satisfies properties 1 and 3 (by a simple change of variable for the latter), but not necessarily 2. 
%\marginpar{\tiny\textcolor{red}{Expliquer pédagogiquement ce qu'on fait...}}
% (whose time-$1$ map $h_\lambda$ satisfies $h_\lambda = \f_\lambda h \f_\lambda^{-1}$). 
Now given $u \in(0, 
%_{[\frac{\lambda}8, 1-\frac{\lambda}8]}
\|X_\lambda\|_0]$, let
\begin{equation}
\label{e:yl}
Y := \rho \, X_\lambda + (1-\rho)u = u+\rho(X_\lambda-u).
\end{equation}
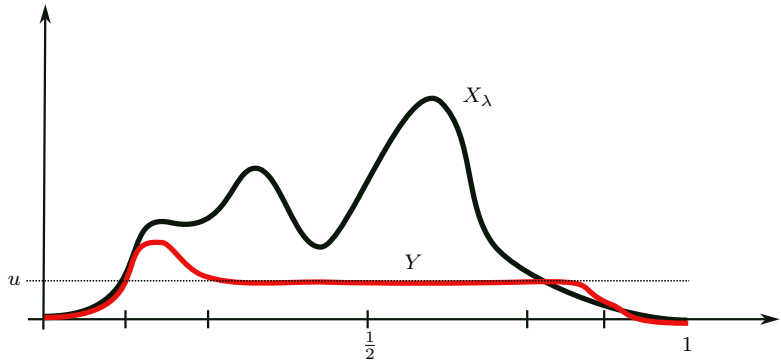
\begin{figure}[ht]
\begin{center}
\includegraphics[scale=1]{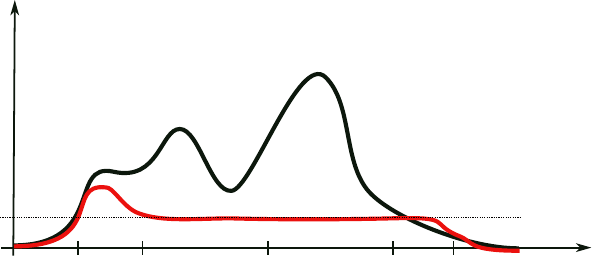}
\put(-142,23){$\scriptstyle{Y}$}
\put(-158,-8){$\scriptstyle{\frac12}$}
\put(-37,-8){$\scriptstyle{1}$}
\put(-292,17){$\scriptstyle{u}$}
\put(-120,86){$\scriptstyle{X_\lambda}$}
\caption{Graphs of $X_\lambda$ and $Y$, for some $u$.}
\label{f:interp}
\end{center}
\end{figure}
This vector field (cf. Figure \ref{f:interp}) still clearly satisfies property 1 above. Now we claim that, for $\lambda$ big enough, whatever the choice of $u$, it also satisfies property 2, and that, for a well chosen $u$, it satisfies property 3.
%\begin{equation}
%\label{e:ineq}
%\int_{x_0}^{1-x_0}\frac1X \ge \int_{\lambda x_0}^{1-\lambda x_0}\frac1Y
%\end{equation}
%for $x_0=\frac1{8\lambda}$. We will then see how to perform a simple modification on $Y_\lambda$ so that it satisfies all three properties.
\medskip

\noindent\textbf{Property 2.} Let us first control $\|X_\lambda\|_0$: 
$$X_\lambda = (\f_\lambda)_*X = (D\f_\lambda\cdot X)\circ \f_\lambda^{-1}$$
so $\|X_\lambda\|_0 =\|D\f_\lambda\cdot X\|_0$. Now on 
$B_\lambda= [0,\frac1{2\lambda}]\cup[1-\frac1{2\lambda},1]$,
%$B_\lambda=[0,\frac1{2\lambda}]\cup[1-\frac1{2\lambda},1]$, 
$$0\le D\f_\lambda\cdot X \le \lambda X \le \lambda \max_{B_\lambda}X=\lambda \cdot o(\tfrac1\lambda)\xrightarrow[\lambda\to+\infty]{}0$$ 
since $X$ is $C^1$ flat (i.e. $X=DX=0$) at $0$ and $1$. Furthermore, on $[0,1]\setminus B_\lambda$, 
$$0\le D\f_\lambda\cdot X \le \frac1{2\lambda-2}\|X\|_0 \xrightarrow[\lambda\to+\infty]{}0.$$
So in the end, $\|X_\lambda\|_0$ goes to $0$ when $\lambda$ goes to infinity.

Now 
$0\le Y \le \|X_\lambda\|_0$ and, 
for every $s\in[\![0,r]\!]$, 
$$D^sY = \sum_{k=0}^s \binom{s}{k}D^k\rho \times D^{s-k}(X_\lambda-u).$$
First note that $0\le|D^0(X_\lambda-u)|\le \|X_\lambda\|_0$. Next,
on $[0,\frac14]\cup[\frac34,1]$, $X_\lambda(x)=\lambda X(\frac x \lambda)$ so for every $l\in[\![1,s]\!]$,
%$$|D^0(X_\lambda-u_\lambda)|=|X_\lambda-u_\lambda|\le $$
$$|D^l(X_\lambda-u)(x)|=\lambda^{1-l}|D^lX(\tfrac x \lambda)|\le \lambda^{1-l} \max_{B_\lambda}|D^lX|,$$
and so
$$|D^sY|\le 2^s \|\rho\|_s\times \max(\|X_\lambda\|_0,\max_{B_\lambda}|DX|,\tfrac1\lambda \|X\|_s)$$
which goes to $0$ when $\lambda$ goes to infinity (for the second term of the parenthesis, this comes from the fact that $DX(0)=DX(1)=0$). Since $Y$ is constant on $[\frac14,\frac34]$, we get that, for $\lambda$ big enough, for any $u\in(0,\|X_\lambda\|_0]$, $Y$ does satisfy property 2.\medskip

\noindent\textbf{Property 3.} By change of variables $x\mapsto \f_\lambda(x)$, for $x_0=\frac1{8\lambda}$,
$$\int_{x_0}^{1-x_0}\frac1X = \int_{\lambda x_0}^{1-\lambda x_0}\frac1{X_\lambda}.$$
Now if we pick $u=\|X_\lambda\|_0$, we have $X_\lambda\le Y$ everywhere, so we get 
$$\int_{x_0}^{1-x_0}\frac1X \ge \int_{\lambda x_0}^{1-\lambda x_0}\frac1{Y}.$$
For any $u$ now, 
$$ \int_{\lambda x_0}^{1-\lambda x_0}\frac1{Y} = \int_{\frac18}^{\frac78}\frac1{Y}\ge \int_{\frac12-\frac{1}{4\lambda}}^{\frac12+\frac{1}{4\lambda}}\frac1{Y}=\frac1{2\lambda u}\xrightarrow[u\to0]{}+\infty.$$
Hence, making $u$ vary continuously between $\|X_\lambda\|_0$ and $0$, one will come across a value for which one has exactly the equality of property 3, which concludes the proof.
%
%%d'ailleurs sans-doute changer l'énoncé.
%
%\bigskip
%
%\textcolor{red}{Il faut réécrire celle de la Proposition 8.3 de \cite{EN21} en ajoutant le paramètre d'ajustement de la zone ``plateau'' pour annuler le shift. A-t-on besoin des formes normales comme je l'avais écrit au brouillon ? J'en doute vu \cite{EN21}, et en plus cela ne nous donnerait pas le résultat en régularité finie alors que la preuve générale nous le donne.}

\subsection{Proof of Proposition \ref{p:interp-ITI}}
\label{ss:interp}

Proposition \ref{p:interp-ITI} directly follows from the combination of Proposition \ref{p:interp} and Lemma \ref{l:smoothing} below, to which this section is devoted. The former is based on Lemma \ref{lem:regularitypoints} below which was proved in~\cite{BE16}.\medskip

Given $f\in\Diff^\infty_+([0,1])$ and $x\in[0,1]$, we let $f^{\pm1}(x)=\max(f(x),f^{-1}(x))$ and $f^{\mp1}(x)=\min(f(x),f^{-1}(x))$.

\begin{lem} \label{lem:regularitypoints}
Consider a diffeomorphism $f \in \mathrm{Diff}_+^{\infty}([0,1])$ which is ITI at $0$ but nowhere in the interior, and which is the time-$1$ map of a $C^1$ vector field $X$ on $[0,1]$ (automatically $C^\infty$ on $(0,1)$). Fix an integer $r \geq 1$ and a real number $\delta >0$. Then there exists $x_0>0$ arbitrarily close to $0$ such that 
%$$ \sup\{|Df(y)-1| ; y\in[x_0, x_0+(f^\pm(x_0)-x_0)^{2\delta}]\}<(f^\pm(x_0)-x_0)^{1-\delta}$$
$$ \left\| X_{|[f^{\mp2}(x_0), f^{\pm2}(x_0)]} \right\|_r \leq (f^{\pm1}(x_0)-x_0)^{1-\delta}$$
and
$$|X(x_0)| \geq \frac{1}{2}(f^{\pm1}(x_0)-x_0).$$
\end{lem}

\begin{prop}[$C^{r}$ interpolation for isolated ITI fixed points]
\label{p:interp}
Let $\eta >0$ and $r \geq 1$. Under the hypothesis of Lemma \ref{lem:regularitypoints}, there exists a non fixed point $x_0\in(0,1-\eta)$ arbitrarily close to $0$ and a $C^1$ vector field $Y:[0,1] \rightarrow \mathbb{R}$ with the following properties.
\begin{enumerate}
\item $Y_{|[0,x_0]}=X_{|[0,x_0]}$.
\item The vector field $Y$ is equal to the constant vector field $X(x_0)$ on $[x_0+\eta,1]$.
\item The vector field $Y$ is $C^\infty$ on $(0,x_0]$ and $[x_0,1)$ and $C^r$ on $(0,1)$.
\item $\left\| Y_{|[f^{\mp1}(x_0),x_0+\eta]} \right\|_r \leq \eta$.
\item The vector field $Y$ does not vanish on $[x_0,x_0+\eta]$.
\end{enumerate}
\end{prop}

The third condition implies in particular that $Y$ has right and left derivatives of any order at~$x_0$. 

\begin{proof}
Let $\chi:[0,+\infty) \rightarrow \mathbb{R}$ be a $C^{\infty}$ function which is equal to $1$ on $[0,\frac{1}{2}]$ and to $0$ on $[1,+\infty)$ with $0 \leq \chi \leq 1$. Let $M_r=\left\| \chi \right\|_r$.
Take $\delta < \frac{1}{2r+1}$. Take $x_0$ such that Lemma \ref{lem:regularitypoints} holds (without loss of generality, assume $f(x_0)>x_0$) and small enough so that
$$(f(x_0)-x_0)^\delta < \eta$$
$$(f(x_0)-x_0)^{1- \delta}+ r2^r M_r (f(x_0)-x_0)^{1-(2r+1)\delta} < \eta$$
$$ \frac{1}{2}(f(x_0)-x_0) >(f(x_0)-x_0)^{1-\delta} (e^{(f(x_0)-x_0)^{2\delta}}-1).$$
%$$(f(x_0)-x_0)^{2\delta}<\frac{x_0^+-x_0}2.$$
For the last inequality, observe that the right-hand side is equivalent to $(f(x_0)-x_0)^{1+\delta}$ when $x_0$ goes to $0$. 
%For the last, observe that by the mean value theorem, 
Take $b=(f(x_0)-x_0)^{2\delta}$.
For any $0 \leq i \leq r$, let $a_i=D^i X(x_0)$.
We define $Y$ by setting $Y(x)=X(x)$ if $x \leq x_0$ and
$$Y(x)=a_0+ \chi \left(\frac{x-x_0}{b} \right)\sum_{i=1}^{r}\frac{a_i}{i!}(x-x_0)^i$$
if $x \geq x_0$.
It is clear that the vector field $Y$ satisfies the three first properties. Let us now deal with property 4.
%\marginpar{\tiny \textcolor{red}{remplacer $h$ par une autre lettre? ($h$ est un difféo comme $f$ ou $g$ pour moi ;)} \textcolor{blue}{ J'ai remplacé par  $y$}}

For any $1\leq i \leq r$, let $f_i$ be defined for $y \in \mathbb{R}$ by
$$f_i(y)=\frac{a_i}{i!} \delta^i \chi \left(\frac{y}{b}\right).$$
Then for any $0 \leq k \leq r$
$$D^kf_i(y)=\sum_{j=0}^{i} \binom{k}{j}  \frac{a_i}{(i-j)!}\frac{y^{i-j}}{b^{k-j}} D^{k-j}\chi \left(\frac{y}{b}\right).$$
Hence, as this function is supported in $[0,b]$,
$$\begin{array}{rcl}

\left\| D^k f_i \right\|_{\infty} & \leq & \displaystyle \sum_{j=0}^{i} \binom{k}{j}  |a_i|b^{i-k} M_r \\
 & \leq & 2^{k} M_r (f(x_0)-x_0)^{1-(2k+1)\delta} \\
 & \leq & 2^{r} M_r (f(x_0)-x_0)^{1-(2r+1)\delta}.
\end{array}$$
From this and from our choice of $x_0$ we deduce the fourth point.\medskip

It remains to check the last point. If $Y$ vanishes at some point $x_1$, then $x_1=x_0+y_1$ belongs to $[x_0,x_0+b]$. As $Y(x_1)=0$,
$$ a_0=-\sum_{i=1}^{k} f_i(y_1)$$
so that
$$\begin{array}{rcl}
\frac{1}{2}(f(x_0)-x_0) & \leq & |a_0| \\
 & \leq & \displaystyle \sum_{i=1}^{k} |f_i(y_1)| \\
 & \leq & \displaystyle \sum_{i=1}^{k} \frac{1}{i!} (f(x_0)-x_0)^{1+(2i-1) \delta} \\
 & \leq & (f(x_0)-x_0)^{1-\delta} (e^{(f(x_0)-x_0)^{2\delta}}-1),
\end{array}$$ 
which is not possible by our choice of $x_0$.

\end{proof}

\begin{lem}[Smoothing]
\label{l:smoothing}
Let $x_0 \in (0,1)$, $r \geq 1$ and $\alpha\in(0,1-x_0)$. Let $Y$ be a $C^r$ vector field on $(0,1)$ which is $C^{\infty}$ on $(0,x_0]$ and on $[x_0,1)$. Then there exists a $C^{\infty}$ vector field $Z$ on $(0,1)$ such that $Z=Y$ outside $(x_0,x_0+\alpha)$ and $\left\|Z-Y \right\|_r < \alpha$.
\end{lem}

The proof of this lemma is an adaptation of the proof of Borel's theorem about the existence of a $C^{\infty}$ function with prescribed derivatives of any order at a point.

\begin{proof}
We use the same function $\chi$ as in the previous proof. For any $k$, we let $M_k=\left\| \chi \right\|_k$.
For $i >r$, let
$$a_i= \lim_{y \rightarrow 0^{+}} D^i Y(x_0-y)-D^i Y(x_0+y)$$
and
$$g_i(v)=\frac{a_{i}}{i!}y^{i}\chi\left(\frac{y}{b_i}\right).$$
We will show that, for well chosen $b_i$'s and for any $k$, the series of the $k$th derivatives of the $g_i$'s converge uniformly so that $\sum_{i=r+1}^{+\infty}g_i$ defines a $C^{\infty}$ function with $C^{r}$ norm smaller than $\alpha$ and support in $[0,\alpha]$.

Take $b_i=\frac{1}{2}$ if $a_i=0$ and
$$b_i= \frac{1}{2} \min \left(\frac{\alpha}{|a_i|2^{r}M_r},\alpha \right)$$
otherwise. For any $k \leq i$ and $v \in [0,1]$, 
$$\begin{array}{rcl}
|D^k g_i(y)| & = & \displaystyle \left| \sum_{j=0}^{k} \binom{k}{j}\frac{a_i y^{i-j}}{(i-j)! b_{i}^{k-j}}D^{k-j}\chi\left(\frac{y}{b_i}\right) \right| \\
& \leq & 2^{k} |a_i| |b_i|^{i-k} M_k \\
& \leq & 2^{k-r}  \frac{\alpha}{2^{i-k}} \frac{M_{k}}{M_r}.
\end{array}$$
This proves that the series $\displaystyle \sum_i D^{k} g_i$ converges uniformly. Let
$$ z(y)= \sum_{i=r+1}^{+\infty}g_{i}(y).$$
As a consequence of the above estimates, $z$ is well defined on $[0,1]$, $C^{\infty}$ and $\left\| z\right\|_r \leq \alpha$. Moreover, $z$ is supported in $[0,\alpha]$ and satisfies, for any $k \geq r+1$, $D^{k}z(0)=a_i$.
It suffices then to take the vector field $Z$ on $(0,1)$ defined by
$$Z(x)=\left\{ \begin{array}{lll} Y(x) & \text{if} & x \leq x_0 \\
  Y(x)+z(x-x_0)&\text{if} &x \geq x_0.
\end{array} \right.$$
\end{proof}

\subsection{Proof of Proposition \ref{p:reduc}}
\label{ss:reduc}

In this subsection, we derive Proposition \ref{p:reduc} from Lemma \ref{l:woITI} and Proposition \ref{p:interp-ITI}.\medskip

%\marginpar{\textcolor{red}{Trouver d'autres noms pour $\chi$ et $\rho$ qui désignent autre chose que dans les preuves précédentes...}}
 The case without ITI fixed point on the boundary (and thus with a finite number of non ITI fixed points) has been dealt with in Section \ref{ss:easy-case}. So we now focus on the case where $f$ is a smooth diffeomorphism of $[0,1]$ ITI at  some boundary point, say $0$ (and possibly at $1$) but nowhere in the interior, and generated by a $C^1$ vector field $X$. We will treat the case where $f$ is also ITI at~$1$ (which is the most difficult and contains all the ideas), and leave the other case as an exercise. \medskip
 
Let $\eps>0$ and $r\in\N$, and let $\eta>0$ be as in Lemma \ref{l:flow} with respect to these $\eps$ and $r$. We want to build a vector field $\tilde X$ smoothly conjugated to $X$ and satisfying $\|\tilde X\|_r<\eta$ (it will, in addition, coincide with $X$ near the endpoints and the conjugacy will coincide with the identity near $0$ and with a power of $f$ near $1$). \medskip
 
 \noindent\textbf{Preparation near the boundary.}   Fix a step function $\chi$ equal to $0$ on $[0,\frac38]$ and $1$ on $[\frac58,1]$, a bump function $\rho$ equal to $0$ outside $[\frac14,\frac34]$ and to $1$ on $[\frac5{16},\frac{11}{16}]$ and let 
 \begin{equation}
 \eta'=\frac\eta3 \min\left(1,\frac1{\|\rho\|_r},\frac1{2\|\chi\|_r}\right).
 \end{equation} 
%{\color{blue} (Beware that $\chi$ and $\rho$ are not exactly the same step and bump functions as in the previous sections)} \marginpar{\tiny \color{blue}{Je propose de supprimer cette parenthèse. Je pense que c'est suffisamment clair}}. 
Applying Proposition~\ref{p:interp-ITI} to $X$ with this $\eta'$ instead of $\eta$, we get a point $x_0$ and a new vector field $Y$ with properties 1 to 5 of Proposition \ref{p:interp-ITI}. One may assume that $x_0+\eta'<\frac14$, and that $f$ is $\eps$-$C^r$-close to the identity on $[0,f^{\pm1}(x_0)]$. \medskip

Exchanging the roles of $0$ and $1$, another application of Proposition \ref{p:interp-ITI} gives us a point $x_1$, with $\frac34<x_1-\eta'$ and $f$ $\eps$-$C^r$-close to the identity on $[f^{\mp1}(x_1),1]$, and a vector field $Z$ satisfying the following analogs of properties 1 to 5:
\begin{itemize}
\item[1'.] $Z_{|[x_1,1]}=X_{|[x_1,1]}$.
\item[2'.] The vector field $Z$ is equal to the constant vector field $X(x_1)$ on $[0,x_1-\eta']$.
\item[3'.] The vector field $Z$ is $C^\infty$ on $(0,1)$,
\item[4'.] $\left\| Z_{|[x_1-\eta',f^{\pm1}(x_1)]} \right\|_r \leq \eta'$.
\item[5'.] The vector field $Z$ does not vanish on $[x_1-\eta',x_1]$.
\end{itemize}

\noindent\textbf{Case 1: $f$ has no fixed point between $x_0$ and $x_1$}.\medskip

 Define a vector field $\tilde X$ which coincides with $Y$ on $[0,\frac14]$, with $Z$ on $[\frac34,1]$ and such that, for $x\in[\frac14,\frac34]$, 
$$\tilde X(x) = X(x_0)+ \chi(x) (X(x_1)-X(x_0))\quad \text{(cf. Figure \ref{f:Xtilde})}.$$
\begin{figure}[ht]
\begin{center}
\includegraphics[scale=1]{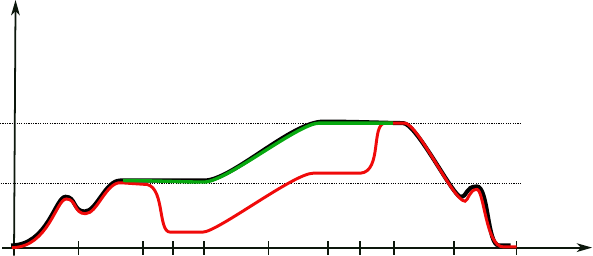}
\put(-142,23){$\scriptstyle{\tilde X_u}$}
\put(-158,-8){$\scriptstyle{\frac12}$}
\put(-98,-8){$\scriptstyle{\frac34}$}
\put(-219,-8){$\scriptstyle{\frac14}$}
\put(-206,-8){$\scriptstyle{\frac5{16}}$}
\put(-189,-8){$\scriptstyle{\frac38}$}
\put(-115,-8){$\scriptstyle{\frac{11}{16}}$}
\put(-130,-8){$\scriptstyle{\frac58}$}
\put(-37,-8){$\scriptstyle{1}$}
\put(-309,32){$\scriptstyle{X(x_0)}$}
\put(-309,61){$\scriptstyle{X(x_1)}$}
\put(-120,71){$\scriptstyle{\tilde X}$}
\caption{Graphs of $\tilde X$ and $\tilde X_u$, for some $u$.}
\label{f:Xtilde}
\end{center}
\end{figure}
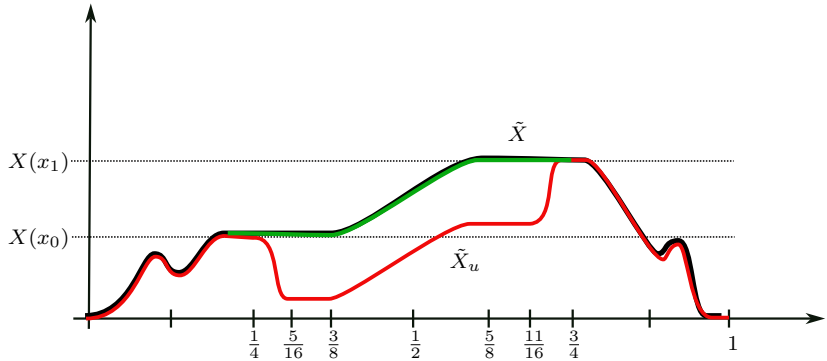

By construction, $\tilde X$ is $C^\infty$ on $(0,1)$.
To fix ideas, let us assume from now on that $0<X(x_0)\le X(x_1)$, so that $\tilde X$ increases from $X(x_0)>0$ to $X(x_1)$ on $[x_0+\eta',x_1-\eta']\supset[\frac14,\frac34]$. In particular, if $a=\max\{x\le x_0;X(x)=0\}$ and $b=\min\{x\ge x_1, X(x)=0\}$, $\tilde X$ vanishes neither on $(a,x_0]\cup[x_1,b)$ where it coincides with $X$, nor on $[x_0,x_0+\eta']\cup[x_1-\eta',x_1]$ according to properties 5 and 5', nor on $[x_0+\eta',x_1-\eta']$. 
Hence, by Lemma \ref{l:crit} (applied to the time-$1$ maps of $X$ and $\tilde X$), $\tilde X$ is the push-forward of $X$ on $[a,b]$ by a $C^1$ diffeomorphism coinciding with $\id$ on $[a,x_0]$ and with some flow map $f^\sigma$ of $X$ on $[x_1,b]$ (and this conjugacy extends to $[0,1]$ by $\id$ on $[0,a]$ and $f^\sigma$ on $[b,1]$). 

To conclude the proof of Proposition \ref{p:reduc} in this case, we need to modify a little the definition of $\tilde X$ so that this $\sigma$ is an integer, and to check that its time-$1$ map is $\eps$-$C^r$-close to the identity. \medskip

\noindent\textbf{Adjustment of $\sigma$.} Note that, by Lemma \ref{l:crit}, this $\sigma$ is simply given by 
$$\sigma = \int_{x_0}^{x_1}\frac1{X} - \int_{x_0}^{x_1}\frac1{\tilde X}$$
%{\color{blue}(this is the difference between the times needed to go from $x_0$ to $x_1$ ``at speed given by $X$ or $\tilde X$'')} \marginpar{\tiny \color{blue} Je propose de supprimer ça quitte à insister plus lourdement là-dessus quand on énonce le lemme \ref{l:crit}}. 
We thus need to modify $\tilde X$ so as to keep all its good properties but ensure in addition that the above number is an integer. 
For this, let us consider a whole family $\tilde X_u$ coinciding with $\tilde X$ on $[0,\frac14]\cup[\frac34,1]$ and such that, for $x\in[\frac14,\frac34]$, 
$$\tilde X_u(x) = \tilde X(x)
%X(x_0)+ \chi(x) (X(x_1)-X(x_0))
-u\rho(x)X(x_0)\quad \text{(cf. Figure \ref{f:Xtilde})}.$$

%\marginpar{\tiny \color{blue} N'étant pas fans des parenthèses, je les ai supprimées dans ce paragraphe.}
For $u\in[0,1)$, these smooth vector fields, which coincide with $\tilde X$ outside $[\frac14,\frac34]$, still have no fixed point in $[\frac14,\frac34]$ since $\tilde X$ is bounded below by $X(x_0)$ there. However, by definition of $\chi$ and $\rho$, on $[\frac5{16},\frac38]$, $\tilde X_u(x)=X(x_0)(1-u)$, so when $u$ goes to $1$, $\sigma = \int_{x_0}^{x_1}(\frac1X-\frac1{\tilde X_u})$ goes to $-\infty$ in a continuous way, 
% the length of the orbit segments of $\tilde X_u$ contained in $[\frac14,\frac34]$ go to infinity, 
so there are infinitely many choices of $u$ for which $\sigma$ is an integer. \medskip

\noindent\textbf{$C^r$-control on the time-$1$ map $\tilde f$.} Recall we have assumed, to fix idea, that $X>0$ on $[x_0,x_1]$ so $f>\id$ there. On $[0,x_0]$ (resp. $[x_1,1]$), $\tilde X_u = X$, so on $[0,f^{-1}(x_0)]$, (resp. $[f(x_1),1]$), $f=\tilde f$, which was assumed $\eps$-$C^r$-close to the identity there. By definition of $\eta$, it remains to check that $\|\tilde X_u\|_r<\eta$ on $[x_0,f(x_1)]$. Now by definition, on this interval,
$$\|\tilde X_u\|_r \le \|\tilde X\|_r + u |X(x_0)|\cdot\|\rho\|_r.$$
By property 4 of $Y$, $|X(x_0)|=|Y(x_0)|\le \eta'$ so by choice of $\eta'$, the second summand is smaller than $\eta/3$. As for the first one, still on the same interval,
$$\|\tilde X\|_r \le |X(x_0)|+|X(x_1)-X(x_0)|\cdot\|\chi\|_r \le \eta'+2\eta' \|\chi\|_r\le \frac{2\eta}3$$
again by the choice of $\eta'$, which gives the desired control on $\tilde X_u$ and thus concludes the proof in \emph{Case 1}.\medskip

\noindent\textbf{Case 2: $f$ does have fixed points between $x_0$ and $x_1$}.\medskip
%\marginpar{\tiny \color{blue} J'ai supprimé des parenthèses dans ce paragraphe. J'ai aussi remplacé des "fixed points" par "zeroes".}
\noindent\textbf{General idea.} The ``preparation'' is the same as in case 1, yielding vector fields $Y$ and $Z$, but this time we cannot simply connect them monotonously from $(x_0+\eta,X(x_0))$ to $(x_1-\eta,X(x_1))$ because, roughly, unlike in the previous case, the restrictions to $[x_0,x_1]$ of $X$ and the resulting vector field $\tilde X$ would not be conjugate: the former would have zeroes while the latter wouldn't. 

Let $c$ be the smallest zero of $X$ greater than $x_0$ and $d$ the biggest one smaller than $x_1$. 
Assume for a while that $x_0+\eta<c-\eta<d+\eta<x_1-\eta$ (1) and that $X$ is already $C^r$-small near $[c,d]$ (2). Then one could hope to interpolate between $Y$ on $[0,x_0+\eta]$ and $X$ near $c$, and between $X$ near $d$ and $Z$ on $[x_1-\eta,1]$. But the resulting vector field won't necessarily be itself $C^r$-small on $[x_0+\eta,x_1-\eta]$ unless the room available to interpolate is sufficient (3) compared to the control we have on $X$. The argument below is designed precisely to take care of points (1), (2), (3) (cf.~Figure \ref{f:Xtildeb} for an illustration of the process).   \medskip

\begin{figure}[htbp]
\begin{center}
\includegraphics[scale=1]{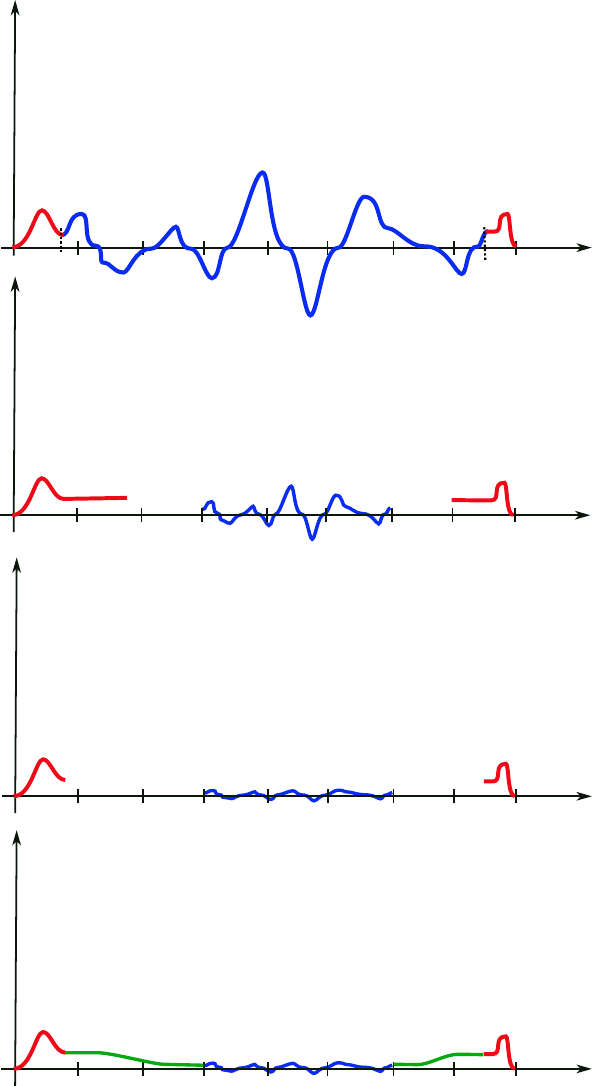}
\put(-53,392){$\scriptstyle{x_1}$}
\put(-59,407){$\scriptstyle{d}$}
\put(-240,407){$\scriptstyle{c}$}
\put(-260,392){$\scriptstyle{x_0}$}
\put(-130,422){$\scriptstyle{X}$}
\put(-40,26){$\scriptstyle{\tilde X}$}
\put(-142,150){$\scriptstyle{\bar X}$}
\put(-142,295){$\scriptstyle{h_* X}$}
\put(-262,295){$\scriptstyle{Y}$}
\put(-38,295){$\scriptstyle{Z}$}
\put(-158,-8){$\scriptstyle{\frac12}$}
\put(-98,-8){$\scriptstyle{\frac34}$}
\put(-219,-8){$\scriptstyle{\frac14}$}
%\put(-206,-8){$\scriptstyle{\frac5{16}}$}
\put(-189,-8){$\scriptstyle{\frac38}$}
%\put(-115,-8){$\scriptstyle{\frac{11}{16}}$}
\put(-130,-8){$\scriptstyle{\frac58}$}
\put(-37,-8){$\scriptstyle{1}$}
%\put(-309,32){$\scriptstyle{X(x_0)}$}
%\put(-309,61){$\scriptstyle{X(x_1)}$}
%\put(-120,71){$\scriptstyle{\tilde X}$}
\caption{From the graph of $X$ to that of $\tilde X$.}
\label{f:Xtildeb}
\end{center}
\end{figure}
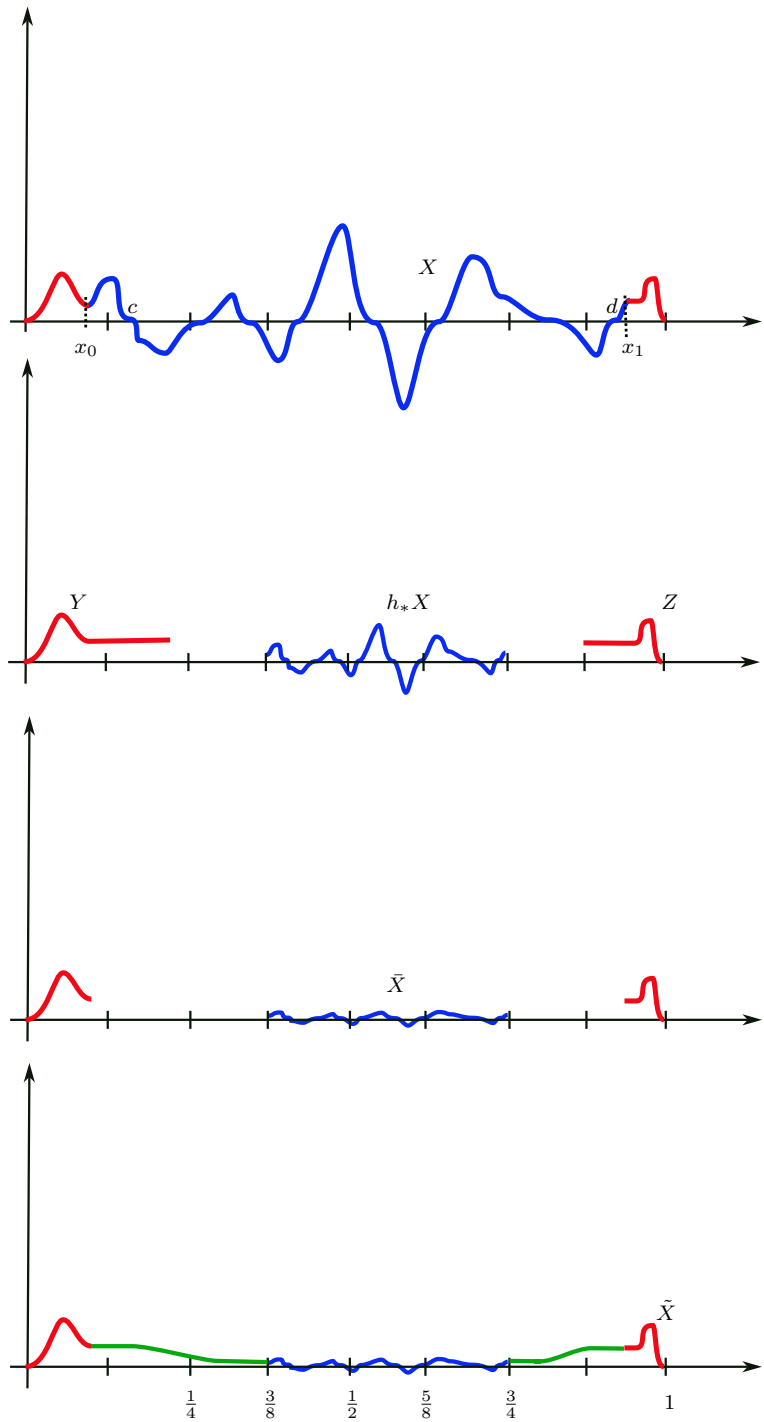

Let $h:[x_0,x_1]\to [\frac38,\frac58]$ be a homothety of ratio denoted by $\mu$ and consider the smooth vector field $h_*X$ on $[\frac38,\frac58]$, which vanishes on $\{c',d'\}=\{h(c),h(d)\}$ and has no zero outside $[c',d']$.
%, \textcolor{red}{and whose time one flow has no ITI fixed point at all (pourquoi cette précision ?)}{\color{blue} On peut supprimer. Dans la version encore avant on parlait de "ITI fixed point" du champ de vecteurs ce qui n'était pas correct donc j'ai rajouté la mention du flot.} 

Applying Proposition \ref{l:woITI} to ${h_*X}_{|[c',d']}$, we get a smooth vector field $\bar X$ on $[c',d']$ which is $\eta'$-$C^r$-small and conjugated to ${h_*X}_{|[c',d']}$ by a diffeomorphism which is a homothety of some ratio $\lambda$ near $c'$ and $d'$. Actually, proceeding just like in the proof of Proposition~\ref{l:woITI}, we can get slightly more : an $\bar X$ on $[\frac38,\frac58]$ which is $C^r$-$\eta'$-small, constant near $\frac38$ and $\frac58$ and smoothly conjugated to $h_*X$ near $[c',d']$ by a diffeomorphism which is a homothety of some ratio $\lambda$ near $c'$ and $d'$. %
%, one can build a new smooth vector field $\bar X$ on $[\frac38,\frac58]$ which is $C^r$-$\eta'$-small, constant near $\frac38$ and $\frac58$ and smoothly conjugated to $h_*X$ near $[h(a),h(b)]$ by a diffeomorphism which is a homothety of some ratio $\lambda$ near $h(a)$ and $h(b)$.  

At this point, we have basically reduced to the case where (1), (2) and (3) are satisfied.
Namely, we can now use a step function $\bar\chi$ (just like in case 1) to define a vector field $\tilde X$ which coincides with $Y$ on $[0,\frac14]$, $\bar X$ on $[\frac38,\frac58]$ and $Z$ on $[\frac34,1]$, and which is still $\eta$-$C^r$-small provided $\eta'$ was chosen small enough in terms of $\|\bar\chi\|_r$. We already know that the restriction of $\tilde X$ to $[c',d']$ is conjugated to the restriction of $X$ to $[c,d]$ by a diffeomorphism which coincides with a homothety of ratio $\mu\lambda$ near the boundary. Furthermore, by the criterion of conjugacy (Lemma~\ref{l:crit}), one easily checks that
\begin{itemize}
\item the restriction of $\tilde X$ to $[0,c']$ is conjugated to the restriction of $X$ to $[0,c]$ by a diffeomorphism which coincides with a homothety of ratio $\mu\lambda$ near $c$ and with some flow-map $f^\sigma$ of $X$ near $0$;
\item its restriction to $[d',1]$ is conjugated to the restriction of $X$ to $[d,1]$ by a diffeomorphism which coincides with a homothety of ratio $\mu\lambda$ near $d$ and with some flow-map $f^\tau$ near $1$.
\end{itemize}

These three conjugacies glue up to give a full conjugacy between $X$ and $\tilde X$, but as before, we would like $\sigma$ and $\tau$ to be integers. This can be ensured by modifying slightly the definition of $\tilde X$ just like in case 1, more precisely by bringing $\tilde X$ closer to $0$ in the ``green region'' of Figure \ref{f:Xtildeb}, as we did for the green region of Figure \ref{f:Xtilde}. The proof of the $C^r$-$\eps$-closeness to $\id$ of the time-1 map of $\tilde X$ works just like in case 1 as well.\medskip

%\noindent\textbf{Case 2.} $f$ is not ITI at $1$, and thus has a finite number of non-ITI fixed points in $[x_0,1]$. This case can be settled with a variation of the arguments of case 1.b. and is left to the reader. 

\subsection{Nonzero rational rotation number} \label{ss:rational}

\begin{prop} 
\label{p:cases}
Let $f$ be a smooth circle diffeomorphism with rotation number $\frac p q\in\Q/\Z$, with $p$ and $q$ coprime, and with vanishing asymptotic variation.  Then:
\begin{enumerate}
\item if $f^q$ has no ITI fixed point, $f$ is smoothly conjugated to a diffeomorphism of the form $g\circ R_{p/q}$ where $g$ is a smooth diffeomorphism fixing $0$, without ITI fixed point, commuting with $R_{p/q}$ and with vanishing asymptotic variation;
\item if $f^q$ has an ITI fixed point, $f$ is smoothly conjugated to a diffeomorphism of the form $h\circ R_{p/q}$ such that $h=\id$ except on $R_{-\frac p q}([0,\frac1q])$ and $h$ has vanishing asymptotic variation.
\end{enumerate}
\end{prop}

\begin{proof}[Proof of Proposition \ref{p:cases}] \emph{Case 1}. The diffeomorphism $f^q$ is a circle diffeomorphism with fixed points, all non-ITI, so they are in finite number. We claim that it has a smooth $q$th root with the same fixed points ($f$ is a $q$th root but without fixed point). Indeed, given a fixed point $a$ of $f^q$, if $a_-$ and $a_+$ are the fixed points of $f^q$ closest to $a$ on both sides (which may coincide if $q=2$), it follows from Takens \cite{Ta} that, on $(a_-,a_+)$, $f^q$ is the time-$1$ map of a unique smooth vector field. Now recall we assumed the asymptotic variation of $f^q$ vanishes. According to \cite[Theorem A]{EN21}, this implies that the aforementioned vector fields coincide on the intersection of their domains of definition, so that $f^q$ is the time-$1$ map of a smooth vector field $X$ on the whole circle. Taking $g$ as the time-$1/q$ map $\phi_X^{\frac{1}{q}}$ of this vector field, we get the desired $q$th root, whose fixed points are exactly those of $f^q$ and are not ITI either.

Let us prove that $g$ commutes with $f$. Indeed, we claim first that, if $f^q$ has a $q$th root $g_1$ with the same set of fixed points, the diffeomorphism $g_1$ is equal to $g$. Take a point $a$ of the circle which is not fixed and denote by $(a_-,a_+)$ the connected component of the complement of the set of fixed points which contains the point $a$. By transitivity of the flow on $(a_-,a_+)$, there exists $t_0$ such that $g_1(a)=\phi^{t_0}_X(a)$. But then, the diffeomorphism $g_1\circ\phi^{-t_0}_X$ commutes with $f^q$ and fixes the point $a$. By Kopell's Lemma, 
% (see \cite{Ko})
the diffeomorphism $g_1$ has to coincide with $\phi^{t_0}_X$ on $(a_-,a_+)$. As the diffeomorphism $g_1$ is a $q$-th root of $f$, $g_1=\phi^{\frac{1}{q}}_X=g$ on any component of the complement of fixed points, hence on the whole circle. So $fgf^{-1}$, which satisfies $(fgf^{-1})^q=f f^q f^{-1}=f^q$ and fixes the periodic points of $f$, just like $f^q$ and $g$, must coincide with $g$.

Now let $h=f\circ g^{-1}$. Note that $g$ commutes with $h$, and that $h$ has the same periodic orbits as $f$, so that, in particular, they have the same rotation number. Furthermore, $h^q=f^qg^{-q}=\id$, so $h$ is smoothly conjugated to the rotation $R_{p/q}$. Denoting by $\tilde f$ and $\tilde g$ the conjugates of $f$ and $g$ by the corresponding conjugation, the relation $h=fg^{-1}$ becomes $R_{p/q}=\tilde f \circ \tilde g^{-1}$, so $\tilde f=R_{p/q}\circ \tilde g$, and $\tilde g$ commutes with $R_{p/q}$ and has no ITI fixed point, as required. \medskip

\noindent \emph{Case 2.} Without loss of generality, we can assume that $0$ is an ITI fixed point of $f^q$ and that its orbit under $f$ is the same as its orbit under the rotation $R_{p/q}$, that we will simply denote by $R$ in what follows. Let $I=[0,\frac1q]$ and let us show that, up to a smooth conjugacy, one can assume that
\begin{equation}
\label{e:R1n}
f = R \text{ on } \mathbb{S}^1\setminus f^{-1}(I)\quad \text{and}\quad f 
= f^q\circ R
\text{ on }f^{-1}(I)=R^{-1}(I).
\end{equation} 
Indeed, if $u\in[\![1,q-1]\!]$ is such that $f^u(I)=R^u(I)=[\frac1q,\frac2q]$, we claim that any smooth diffeomorphism $\psi$ of $I=[0,\frac1q]$ which is ITI at $0$ and whose derivatives at $\frac1q$ coincide with that of $f^u$ at $0$ extends in a unique way to a smooth conjugacy $\f$ with the desired behavior. 
The requirement that $\f^{-1}f\f=R$ outside $f^{-1}(I)$, 
or equivalently that $\f = f \f R^{-1}$ outside $I$, 
determines $\f$ unambiguously (step by step) outside $I$ from its restriction $\psi$ to $I$ 
(namely, for $k\in[\![0,q-1]\!]$, $\f= f^k\psi R^{-k}$ on $R^k(I)$). This defines a piecewise smooth homeomorphism, which induces a smooth diffeomorphism of $R^k(I)$ for every $k\in[\![0,q-1]\!]$. 
If we let $\tilde f = \f^{-1}f\f$ on the whole circle, by construction, $\tilde f=R$ outside $f^{-1}(I)=f^{q-1}(I)$ and, on this interval,
$$ \tilde f 
= \psi^{-1} f ( f^{q-1}\psi R^{-q+1})
%= \f^{-1} f^{q}\f R_{\frac1q} 
= (\f^{-1}f^q\f )R 
=\tilde f^q R$$
as required. Note that the ITI hypothesis on $f^q$ implies that $\tilde f$ is smooth.

We still need to check that $\f$ is a smooth diffeomorphism. 
To do this, first note that the boundary condition on $\psi$ ensures that the restrictions of $\f$ to $[0,\frac1q]$ 
and $[\frac1q,\frac2q]$ glue in a smooth way at $\frac1q$. 
Indeed, if $J_-$ and $J_+$ denote the left and right infinite jets, then 
$$J_+\f(\tfrac1q)=J_+(f^u \psi R^{-u})(\tfrac1q)=J_+(f^u\psi)(0)=J_+(f^u)(0)=J_-\psi(\tfrac1q)=J_-\f(\tfrac1q).$$
Now for every $k\in[\![0,q-1]\!]$, near $R^k(\frac1q)$, $\f=f^k\f\tilde f^{-k}$, and $\tilde f^{-k}$, $\f$ and $f^k$ are smooth near $R^k(\frac1q)$, $\tilde f^{-k}(R^k(\frac1q))=\frac 1q$ and $\f(\frac1q)=\frac1q$ respectively, so $\f$ is smooth near $R^k(\frac1q)$, with positive first derivative, which concludes the proof that $\f$ defines a smooth diffeomorphism of the circle.
\end{proof}

\begin{proof}[Proof of extended Theorem \ref{t:reduc2}] According to Proposition \ref{p:cases}, we only need to show that in each of the following cases, $f$ is quasi-reducible:
\begin{enumerate}
\item $f=R_{p/q}\circ g=g\circ R_{p/q}$ with $g$ a smooth diffeomorphism fixing $0$, with vanishing asymptotic variation and without ITI fixed point;
\item $f=h\circ R_{p/q}$ with $h$ a smooth diffeomorphism with vanishing asymptotic variation and equal to $\id$ outside $R_{-p/q}([0,\frac1q])$ .
\end{enumerate} 
In the first case, according to Theorem \ref{t:reduc3}, $g_{|[0,\frac1q]}$ can be conjugated arbitrarily close to the identity by a smooth diffeomorphism $\psi$ of $[0,\frac1q]$ which coincides with a homothety of the same ratio near $0$ and $\frac1q$. Then $\psi$ extends in a unique way to a smooth diffeomorphism $\f$ of the circle commuting with $R_{p/q}$, and $\f f \f^{-1}=R_{p/q}\f g \f^{-1}$, where $\f g \f^{-1}$ is close to the identity on $[0,\frac1q]$ and commutes with $R_{p/q}$ so is close to the identity everywhere, which concludes the proof in this case.\medskip

In the second case, according to Theorem \ref{t:reduc3}, $h_{|[0,\frac1q]}$ can be conjugated arbitrarily close to the identity by a smooth diffeomorphism $\psi$ of $[0,\frac1q]$ which is ITI at $0$ and $\frac1q$. Again, $\psi$ extends in a unique way to a smooth diffeomorphism $\f$ of the circle commuting with $R_{p/q}$, and $\f f \f^{-1}=\f h \f^{-1}R_{p/q}$, where $\f h \f^{-1}$ is close to the identity on $[0,\frac1q]$ and equal to the identity elsewhere, which concludes the proof.\medskip

\end{proof}

\section{Local perfection results}
\label{s:loc-perf}

This section is devoted to the proof of a key ingredient of Theorem \ref{Thm:conjugacyanddistortion}, namely, the ``local perfection'' of the topological groups it involves.

\subsection{Local fragmented perfection for diffeomorphisms of the circle}

%{\color{red}Il faudrait expliquer pourquoi on reprouve un théorème de qqn d'autre. Parce qu'il n'a pas été publié?} {\color{blue} La principale raison pour laquelle on refait cette démo, c'est que l'originale est mal écrite : quand on l'avait lu ensemble, on s'était dit qu'il fallait le faire. Mais je ne sais pas si on veut le dire noir sur blanc dans l'article.}

%\marginpar{\tiny\color{red} Je suis perturbée par l'usage des primes pour désigner autre chose que des dérivées, mais ça va faire bcp de changements si on remplace... {\color{blue}Finalement, on n'utilise jamais le prime pour les dérivées dans l'article donc je ne pense pas que ce soit gênant. Par contre, c'est plus gênant que $\eta$ désigne des nombres et des difféomorphismes. Je me souviens m'être déjà battu avec les notations en écrivant cette section.}{\color{red} Ok pour les primes. Pour $\eta$ il suffit de changer p. 21. $\varphi$ ça te va ? D'ailleurs moi pour les flots je préfère $\phi$} { \color{blue}Ok, j'ai changé les difféos $\eta$ en $\varphi$ et j'ai remplacer les $\varphi$ par des $\phi$ pour les flots.}}

This section is devoted to the following unpublished theorem by Avila (see \cite{Av08}). We provide a proof for completeness' sake.

Let $J_1$ and $J_2$ be two open intervals of $\mathbb{S}^1$ which cover the circle $\mathbb{S}^1$.

\begin{thm}[Local fragmented perfection] \label{Thm:locfragperf}
For any $\eta'>0$, there exists $\eta>0$ such that, for any  diffeomorphism $f$ in $\mathrm{Diff}^{\infty}_+(\mathbb{S}^1)$ with $d_{\infty}(f,Id) < \eta$, there exist two families of diffeomorphisms $(g_i)_{1 \leq i \leq 4}$ and $(g'_i)_{1 \leq i \leq 4}$ with the following properties.
\begin{enumerate}
\item $g_1,g'_1,g_4,g'_4$ are supported in $J_1$ while $g_2,g'_2,g_3,g'_3$ are supported in $J_2$.
\item For any $i$, $d_{\infty}(g_i,Id) <\eta'$ and $d_{\infty}(g'_i,Id) <\eta'$.
\item $f=[g_1,g'_1][g_2,g'_2][g_3,g'_3][g_4,g'_4].$
\end{enumerate}
\end{thm}

To prove this theorem, we use the following lemma.

\begin{lem} \label{Lem:compactlysupportedcommutators}
Let $J$ be an open interval whose closure is contained in $(0,1)$. There exist vector fields $X$ and $Y$ supported in $(0,1)$ such that, for any sufficiently small $t>0$,
\begin{enumerate}
\item for any $x$ in $J$, $\frac{d}{dt}\left([\phi^t_X,\phi^{t}_Y](x) \right)>0$;
\item for any $x$ in $(0,1)$, $\frac{d}{dt}\left([\phi^t_X,\phi^{t}_Y](x) \right) \geq 0$.
\end{enumerate} 
\end{lem}

\begin{rem} Observe that the derivative at $t=0$ of $[\phi^t_X,\phi^{t}_Y](x)$ has to vanish at any $x$ so that the first above condition can hold only for $t>0$.
\end{rem}

\begin{proof}
A straightforward computation shows that, for any vector fields $X$ and $Y$ of $[0,1]$ and any point $x \in [0,1]$,
$$\frac{d^2}{dt^2}\left( [\phi^t_X,\phi^{t}_Y](x) \right)_{t=0}=2(DX(x)Y(x)-DY(x) X(x)).$$

Hence it suffices to choose the vector fields $X$ and $Y$ in such a way that $(DX) Y- (DY) X>0$ on $J=(a,b)$ and $(DX) Y- (DY) X\geq 0$ elsewhere.
%\marginpar{\tiny {\color{blue} J'ai rajouté un dessin ici, même s'il apparait une page plus bas.}}
To do that (see Figure \ref{f:graphX}), fix real numbers $a_1, b_1$ such that $0<a_1<a<b<b_1<1$. 

Take a vector field $X$ supported in $(a_1,1)$ such that $X \geq 0$ on $[0,1]$, $DX\geq 0$ on $(a_1,b_1)$ and $DX >0$ on $[a,b_1]$.
Take a vector field $Y$ supported in $(0,b_1)$ such that $Y \geq 0$ on $(0,1)$, $DY\leq 0$ on $(a_1,b_1)$  and $DY <0$ on $[a,b_1]$. The vector fields $X$ and $Y$ clearly satisfy the required properties.

\begin{figure}[ht]
\begin{center}
\includegraphics[scale=0.5]{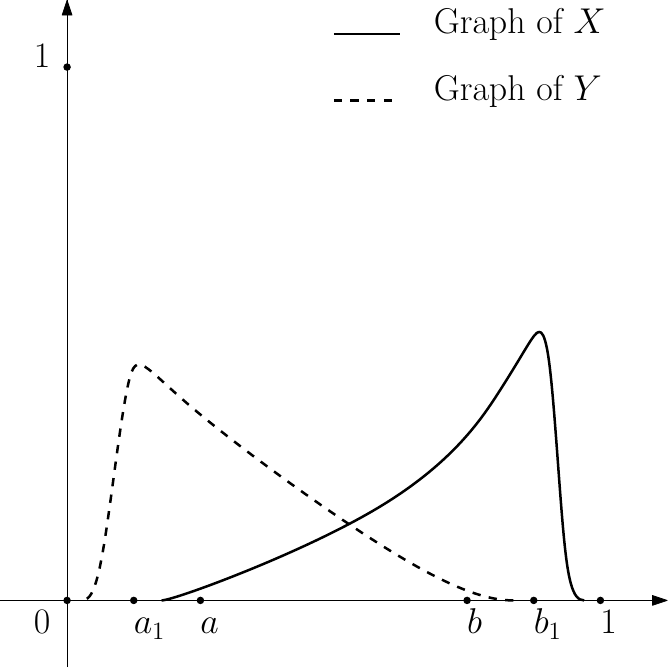}
\caption{Graphs of X and Y}
\label{f:graphX}
\end{center}
\end{figure}

\end{proof}

\begin{proof}[Proof of Theorem \ref{Thm:locfragperf}]
Fix $\eta'>0$. Let $I_1$ be an open interval of $\mathbb{S}^1$ such that $\mathbb{S}^1 \setminus J_2 \subset I_1 \subset J_1$. By Lemma \ref{Lem:compactlysupportedcommutators}, there exist vector fields $X$ and $Y$ which are supported in $J_1$ such that
 for any sufficiently small $t>0$
\begin{enumerate}
\item for any $x$ in $I_1$, $\frac{d}{dt}\left([\phi^t_X,\phi^{t}_Y](x) \right)>0$;
\item for any $x$ in $\mathbb{S}^1$, $\frac{d}{dt}\left([\phi^t_X,\phi^{t}_Y](x) \right) \geq 0$.
\end{enumerate} 

Let $h$ be a $\mathcal{C}^{\infty}$-diffeomorphism of $\mathbb{S}^1$ which sends $J_1$ to $J_2$ such that $I_1$ and $I_2=h(I_1)$ cover the circle. For any $t>0$, let $\xi_t=[\phi^t_X,\phi^{t}_Y] h[\phi^t_X,\phi^{t}_Y]h^{-1}$. Then there exists $\tau>0$ such that, for any $0<t<\tau$ and any point $x$ of the circle, $\frac{d}{dt}(\xi_t(x))>0$ and the diffeomorphisms $\phi^t_X$, $\phi^t_Y$ as well as their conjugates under $h$ are $\frac{\eta'}{2}$-close to the identity. For any $t\in(0,\tau)$, the rotation number of $\xi_t$ is positive. As the rotation number of $\xi_0=Id$ is $0$ and by continuity of the rotation number, there exists $t_0>0$ such that the rotation number of $\xi_{t_0}$ is diophantine.

We then need the following famous theorem due to Herman (see \cite{Her}). For any diffeomorphism $f$ of the circle, we denote by $\rho(f)$ its rotation number.

\begin{thm}[Herman's theorem] \label{Thm:Herman}
Let $\xi$ be a diffeomorphism in $\mathrm{Diff}^{\infty}_+(\mathbb{S}^1)$ with diophantine rotation number. For any $\varepsilon_1>0$, there exists $\varepsilon_2>0$ such that, for any diffeomorphism $g \in \mathrm{Diff}^{\infty}_+(\mathbb{S}^1)$ with $\rho(g)=\rho(\xi)$, and $d_{\infty}(g,\xi)< \varepsilon_2$, there exists a diffeomorphism $h' \in \mathrm{Diff}^{\infty}_+(\mathbb{S}^1)$ such that $d_{\infty}(h',Id)< \varepsilon_1$ and 
$$g=h' \xi (h')^{-1}.$$
\end{thm}

Choose $\varepsilon_1>0$ in such a way that any diffeomorphism $h'$ of the circle which is $\varepsilon_1$-close to the identity sends the union of the supports $\mathrm{supp}(X)\cup\mathrm{supp}(Y)$ inside $J_1$, the set $h(\mathrm{supp}(X)\cup\mathrm{supp}(Y))$ inside $J_2$ and, for any $0<t< \tau$ the conjugates under $h'$ of $\phi^t_X$, $\phi^t_Y$,$h\phi^t_Xh^{-1}$ and $h\phi^t_Y h^{-1}$ are $\eta'$-close to the identity.

Apply Herman's theorem to $\xi=\xi_{t_0}$ with such an $\varepsilon_1>0$.

There exist $0<\tau_1<t_0<\tau_2<\tau$ such that 
$$\rho(\xi_{\tau_1}) < \rho(\xi_{t_0})< \rho(\xi_{\tau_2})$$
and, for any $t \in [\tau_1,\tau_2]$, $d_{\infty}(\xi_t , \xi_{t_0})< \frac{\varepsilon_2}{2}$.

Take $\eta>0$ such that, for any diffeomorphism $f \in \mathrm{Diff}^{\infty}(\mathbb{S}^1)$, which is $\eta$-close to the identity, we have, for any $t \in [\tau_1,\tau_2]$, $d_{\infty}(\xi_t f, \xi_{t_0})< \varepsilon_2$ and 
$$ \rho(\xi_{\tau_1}f)<\rho(\xi_{t_0}) < \rho(\xi_{\tau_2}f).$$ 
By continuity of $\rho$ and the intermediate value theorem, there exists $t_1 \in (\tau_1,\tau_2)$ such that $\rho(\xi_{t_1}f)=\rho(\xi_{t_0})$.  
 
By Herman's theorem, there exists a diffeomorphism $h'$ $\varepsilon_1$-close to the identity such that

$$\xi_{t_1}f=h' \xi_{t_0} (h')^{-1}.$$

Hence
$$f=[g_1,g'_1][g_2,g'_2][g_3,g'_3][g_4,g'_4]$$
with $g_1=h\phi^{t_1}_Y h^{-1}$, $g'_1=h\phi^{t_1}_X h^{-1}$, $g_2=\phi^{t_1}_Y $, $g'_2=\phi^{t_1}_X $,$g_3=h'\phi^{t_0}_X (h')^{-1}$, $g'_1=h'\phi^{t_0}_Y(h')^{-1}$, $g_2=h'h\phi^{t_0}_X (h'h)^{-1}$, $g'_2=h'h\phi^{t_0}_Y (h'h)^{-1}$.
The $g_i$'s and the $g'_i$s are $\eta'$-close to the identity.  
\end{proof}

\subsection{Local perfection for diffeomorphisms of the real line and of the segment}
\label{ss:perfect}
Let $\varepsilon>0$. Fix a distance on $\mathrm{Diff}^{\infty}_c(\mathbb{R})$ which is compatible with the $C^{\infty}$ topology.

\begin{defn}
We call $\varepsilon$-commutator of $\mathrm{Diff}^{\infty}_c(\mathbb{R})$ any diffeomorphism in $\mathrm{Diff}^{\infty}_c(\mathbb{R})$ which is a commutator of elements of $\mathrm{Diff}^{\infty}_c(\mathbb{R})$ which are $\varepsilon$-close to the identity.
\end{defn}

\begin{thm} \label{Thm:localperfectionR}
Let $\varepsilon >0$. Fix a compact interval $I$ of $\mathbb{R}$. There exists $\varepsilon'>0$ such that any diffeomorphism in $\mathrm{Diff}^{\infty}_c(\mathbb{R})$ which is $\varepsilon'$-close to the identity and supported in $I$ can be written as a product of five $\varepsilon$-commutators.
\end{thm}

%\marginpar{\tiny{\color{red} ça me paraît trompeur de le dire comme ça. Il y a aussi le résutlat analogue sur le cercle et le fait de tuer l'invariant de Mather dû essentiellement à Bonatti et al}{\color{blue} ok, j'ai changé la formulation, du coup.}}
The starting point of the proof of this theorem relies on the fact that a diffeomorphism with hyperbolic fixed points is structurally stable on neighborhoods of those fixed points, by a theorem by Sternberg (see \cite{Yoc} appendix 4).

Without loss of generality, we can suppose that $I=[0,1]$.

Fix $\varepsilon>0$. Choose $\varepsilon_1>0$ in such a way that the product of two $C^{\infty}$ diffeomorphisms of $\mathbb{R}$ supported in $[-1,2]$ which are $\varepsilon_1$-close to the identity is $\varepsilon$-close to the identity. 

Take a diffeomorphism $h$ in $\mathrm{Diff}^{\infty}_c(\mathbb{R})$ with the following properties.
\begin{enumerate}
\item $h(0)=0$ and $h(1)=1$.
\item the diffeomorphism $h$ is supported in $[-1,2]$.
\item the diffeomorphism $h$ has no fixed point in $(0,1)$.
\item $Dh(0)>1$ and $Dh(1)<1$.
\item $h$ is generated by a $C^{\infty}$ vector field $X$.
\item For any $\tau \in [-1,1]$, $d_{\infty}(\phi^{\tau}_X,Id_{\mathbb{R}})<\varepsilon_1$.
\end{enumerate}

The following lemma is a consequence of Sternberg's theorem and Lemma \ref{l:crit}.

\begin{lem} \label{Lem:sternberg}
There exists $\varepsilon_2>0$ such that the following property holds. Take any diffeomorphism $u \in \mathrm{Diff}^{\infty}_c(\mathbb{R})$ with the following properties
%\marginpar{\tiny \color{red}{$3\Rightarrow 1$, non?} {\color{blue} Oui, tu as raison. J'ai supprimé l'ancien 1 du coup}}
\begin{enumerate}
\item $d_{\infty}(u,Id)< \varepsilon_2$,
\item the diffeomorphism $u$ is supported in $[0,1]$,
\item the diffeomorphism $uh$ of $[0,1]$ has no interior fixed point and the Mather invariant of $uh_{|[0,1]}$ is trivial.
\end{enumerate}
Then $u$ is an $\varepsilon$-commutator.
\end{lem}

\begin{proof}
Let $\lambda\in(0,\varepsilon_1)$ be a small number which will be fixed later.
By Sternberg's theorem (see \cite{Yoc} appendix 4 for a proof), there exists $\varepsilon_2>0$ such that, for any diffeomorphism $u\in \mathrm{Diff}^{\infty}_c(\mathbb{R})$ supported in $[0,1]$ which is $\varepsilon_2$-close to the identity, the following property holds.
%\marginpar{\tiny{\color{red}$\eta$ désigne des nombres et des difféos dans l'article...}{\color{blue}J'ai changé par des $\varphi$, suivant ta suggestion.}}
There exist $C^{\infty}$-diffeomorphisms $\varphi_0:[0,\frac{3}{4}]=I_0 \rightarrow [0,\varphi_0(\frac{3}{4})] \subset [0,1)$ ITI at $0$ and $\varphi_1:[\frac{1}{4},1]=I_1 \rightarrow [\varphi_{1}(\frac{1}{4}),1] \subset (0,1]$ ITI at $1$ such that, for $i=0,1$,
\begin{enumerate}
\item $\varphi_i^{-1}uh \varphi_i=h$ on $I_i$.
\item $d_{\infty}(\eta_i,Id) < \lambda$.
\end{enumerate}

%\marginpar{\tiny{\color{red}{pas complètement évident de voir que le Lemme 1.3 dit ça...}}{\color{blue} C'est vrai. J'ai mis davantage de détails.}}

Now take a diffeomorphism $u$ as in the statement of the lemma. By Lemma \ref{l:crit}, as the Mather invariants of $h$ and $uh$ are trivial, there exists a $C^{1}$-diffeomorphism $\varphi$ of $[0,1]$ and $\tau \in \mathbb{R}$ such that $\varphi$ coincides with $\varphi_0$ near $0$, with $\varphi_1 \circ \phi_{X}^{\tau}$ near $1$, $\varphi$ is $C^\infty$ on $(0,1)$ and $\varphi^{-1}uh \varphi=h$. Since $\varphi_0$, $\varphi_1$ and $X$ are $C^\infty$, $\varphi$ is actually $C^\infty$ on the closed interval $[0,1]$. The relations $uh \varphi=\varphi h$ and $uh \varphi_i=\varphi_i h$ for $i=0$ and $1$ imply that $\varphi=\varphi_0$ on $[0,\frac{3}{4}]$ and $\varphi=\varphi_1 \circ \phi_{X}^{\tau}$ where $\varphi_1 \circ \phi_{X}^{\tau}$ is defined. Hence $\varphi_0=\varphi_1 \circ \phi_{X}^{\tau}$ where both sides of the equality are defined.

Moreover $\tau=\tau_{uh}(\eta_{0}(\frac{1}{4}),\eta_1(\frac{3}{4}))-\tau_{h}(\frac{1}{4},\frac{3}{4})$. By choosing $\lambda$ and $\varepsilon_2>0$ small enough, we can ensure that $|\tau|\leq 1$.

Extend $\varphi$ as a $C^{\infty}$-diffeomorphism of $\mathbb{R}$ which fixes all the points outside $[0,1]$. Then, as the diffeomorphisms $\varphi_0$, $\varphi_1$ and $\phi^{\tau}_X$ are $\varepsilon_1$-close to the identity, by Property 6 for $\phi_X^\tau$, and by the choice of $\varepsilon_1$, the diffeomorphism $\varphi$ is $\varepsilon$-close to the identity and 
$$\varphi^{-1}uh \varphi=h.$$
Hence $u=\varphi h \varphi^{-1}h^{-1}$ is an $\varepsilon$-commutator.
\end{proof}

To get rid of the condition on the vanishing of the Mather invariant, we need the following lemma.

\begin{lem}[Destruction of the Mather invariant]\label{Lem:destMat}
Let $\varepsilon_3>0$. There exists $\varepsilon_4>0$ such that, for any diffeomorphism $u \in \mathrm{Diff}^{\infty}_c(\mathbb{R})$ supported in $[0,1]$ with $d_{\infty}(u,Id)< \varepsilon_4$ and such that $uh$ has no fixed point in $(0,1)$, there exist two families of diffeomorphisms $(\theta_i)_{1 \leq i \leq 4}$ and $(\theta'_i)_{1 \leq i \leq 4}$ with compact support in $(0,1)$ with the following properties.
\begin{enumerate}
\item The diffeomorphism $[\theta_4,\theta'_4][\theta_3,\theta'_3][\theta_2,\theta'_2][\theta_1,\theta'_1]uh_{|[0,1]}$ has no fixed point in $(0,1)$ and its Mather invariant is trivial.
\item for any $1 \leq i \leq 4$, $d_{\infty}(\theta_i,Id) < \varepsilon_3$ and $d_{\infty}(\theta'_i,Id) < \varepsilon_3$.
\end{enumerate}
\end{lem}

\begin{proof}
Let $p=\frac{1}{2}$ and
$$\begin{array}{rrcl}
\psi: & \mathbb{R} & \rightarrow & [0,1] \\
 & t & \mapsto & \varphi^t_X(p)
 \end{array}
 .$$

For any diffeomorphism $\phi$ which is supported in an open interval $J$ of length one, we denote by $\hat{\phi} $ the diffeomorphism of $\mathbb{R}$ which coincides with $\phi$ on $J$ and commutes with $t\mapsto t+1$  and by $\tilde{\phi}$ the diffeomorphism of the circle $\mathbb{R}/ \mathbb{Z}$ which is induced by $\hat{\phi}$. 

For any $\mathcal{C}^{\infty}$ diffeomorphism $g$ of $[0,1]$ with no interior fixed point and $p \in [0,1]$, recall that $M_g^{p,p}$ is a diffeomorphism of $\mathbb{R}$ which commutes with integral translations. The class of this diffeomorphism modulo pre and post composition by translation is the Mather invariant of $g$. We denote by $\tilde{M}_g^{p,p}$ the diffeomorphism of $\mathbb{S}^1=\mathbb{R}/\mathbb{Z}$ induced by $M_g^{p,p}$. For $\tau \in \mathbb{R}/\mathbb{Z}$, we denote by $R_\tau$ the translation of $\mathbb{S}^1$ by $\tau$.

The following lemma is Corollary 3.7 of \cite{EBN}, which is itself an adaptation of Corollary 1.8 of \cite{BCVW}.

%\marginpar{\tiny{\color{blue}J'ai modifié les notations relatives aux invariants de Mather dans le lemme suivant et la fin de la démonstration qui suit pour être en accord avec les notations de l'introduction.}}
\begin{lem} \label{Lem:fragMat}
Fix a $\mathcal{C}^{\infty}$-diffeomorphism $f$ of $[0,1]$ with no interior fixed point. Let $(\alpha_i)_{1 \leq i\leq \ell}$ be a finite sequence of real numbers with, for any $i$, $\alpha_{i+1}< \alpha_i-1$ and $\alpha_1 \leq 0$. For every $i$, let $\phi_i$ be a diffeomorphism supported in $(\alpha_i-1,\alpha_i)$ and let $h_i=\psi \phi_i \psi^{-1}$. Let $g=f \circ h_1 \circ \ldots \circ h_l$.

Then $\tilde{M}_g^{p,p}\circ R_\tau=\tilde{\phi}_1 \circ \tilde{\phi}_2 \circ \ldots \circ \tilde{\phi}_l \circ \tilde{M}_f^{p,p}$ for some $\tau \in \mathbb{R}$.

\end{lem}

Let $J_1=[\frac{1}{4},1] \subset \mathbb{S}^1= \mathbb{R}/ \mathbb{Z}$ and $J_2=[\frac{3}{4},\frac{1}{2}] \subset \mathbb{S}^1= \mathbb{R}/ \mathbb{Z}$ so that the interiors of $J_1$ and $J_2$ cover the circle.

Choose $\eta'>0$ in such a way that, for any $\mathcal{C}^{\infty}$-diffeomorphism $\lambda$ of $\mathbb{R}$ which is supported in $[-10,0]$ and which is $\eta'$-close to the identity, $\psi \lambda \psi^{-1}$ is $\varepsilon_3$-close to the identity. Let $\eta>0$ be given by Theorem \ref{Thm:locfragperf}. By continuity of the Mather invariant (see the two first pages of Chapter 5 of \cite{Yoc}), there exists $\varepsilon_4>0$ such that, for any diffeomorphism $u$ of $\mathbb{R}$ which is $\varepsilon_4$ close to the identity as in the statement of Lemma \ref{Lem:destMat}, the diffeomorphism $f=(uh)^{-1}$ has no fixed point in $(0,1)$ and the diffeomorphism $\tilde{M}_{f}^{p,p}$ is $\eta$-close to the identity. We apply Theorem \ref{Thm:locfragperf} to $\tilde{M}_f^{p,p}$ to obtain the existence of diffeomorphisms $g_i$ and $g'_i$, for $1 \leq i \leq 4$ which are $\eta'$-close to the identity.

For $1 \leq i\leq 4$ let $\tilde{\phi}_i=[g'_{5-i},g_{5-i}]$ so that $\tilde{\phi}_1\circ \tilde{\phi}_2 \circ \tilde{\phi}_3 \circ \tilde{\phi}_4 \circ \tilde{M}_f^{p,p}$ is the identity. Let $\phi_i$ (respectively $\xi_i$,$\xi'_i$) be the diffeomorphism of $\mathbb{R}$ supported in $[-2i-\frac{7}{4},-2i-1]$ if $i=1,4$ or $[-2i-\frac{5}{4},-2i-\frac{1}{2}]$ if $i=2,3$ which coincides with $\phi_i$ (respectively $g_{5-i}$, $g'_{5-i}$) on its support mod $\mathbb{Z}$. Observe that $\phi_i=[\xi'_i,\xi_i]$ and that both $\xi$ and $\xi'_i$ are $\eta'$-close to the identity . Finally let $h_i=\psi \phi_i \psi^{-1}$, $\theta_i=\psi \xi_i \psi^{-1}$ and $\theta'_i=\psi \xi'_i \psi^{-1}$. Observe that the diffeomorphisms $\theta_i$ and $\theta'_i$ are $\varepsilon_3$-close to the identity.

By Lemma \ref{Lem:fragMat}, the Mather invariant of 
$$f h_1 h_2 h_3 h_4=h^{-1}u^{-1}[\theta'_1,\theta_1][\theta'_2,\theta_2][\theta'_3,\theta_3] [\theta'_4,\theta_4]$$ 
is trivial so that the Mather invariant of its inverse
$$[\theta_4,\theta'_4][\theta_3,\theta'_3][\theta_2,\theta'_2][\theta_1,\theta'_1]uh$$
is trivial.
\end{proof}

\begin{proof}[End of the proof of Theorem \ref{Thm:localperfectionR}]
Take $\varepsilon_2>0$ given by Lemma \ref{Lem:sternberg}. Take $\varepsilon>\varepsilon_3>0$ in such a way that the product of $17$ diffeomorphisms $\varepsilon_3$-close to the identity is $\varepsilon_2$-close to the identity. Let $\varepsilon_4>0$ be given by Lemma \ref{Lem:destMat}. Without loss of generality, we can suppose $\varepsilon_4 \leq \varepsilon_3$.

Take $\varepsilon_4>\varepsilon'>0$ in such a way that the product of $h$ with a diffeomorphism $u$ with support in $[0,1]$ and which is $\varepsilon'$-close to the identity has no fixed point in $[0,1]$. Fix such a diffeomorphism $u$ and apply Lemma \ref{Lem:destMat}: there exist two families of diffeomorphisms $(\theta_i)_{1 \leq i \leq 4}$ and $(\theta'_i)_{1 \leq i \leq 4}$ with compact support in $(0,1)$ with the following properties.
\begin{enumerate}
\item The Mather invariant of $[\theta_4,\theta'_4][\theta_3,\theta'_3][\theta_2,\theta'_2][\theta_1,\theta'_1]uh_{|[0,1]}$ is trivial.
\item for any $1 \leq i \leq 4$, $d_{\infty}(\theta_i,Id) < \varepsilon_3$ and $d_{\infty}(\theta'_i,Id) < \varepsilon_3$.
\end{enumerate}
Let $u'=[\theta_4,\theta'_4][\theta_3,\theta'_3][\theta_2,\theta'_2][\theta_1,\theta'_1]u$. Observe that $u'$ is $\varepsilon_2>0$-close to the identity and that $u'h$ has a trivial Mather invariant. We can hence apply Lemma \ref{Lem:sternberg} to $u'$. The diffeomorphism $u'h$ is a $\varepsilon$ commutator and the diffeomorphism $uh$ is a product of $5$ $\varepsilon$-commutators.
\end{proof}

All the commutators which are constructed in the proof of Theorem \ref{Thm:localperfectionR} preserve the interval $I$. Hence, as a corollary of the proof of Theorem \ref{Thm:localperfectionR}, we obtain the following result of local perfection for $C^{\infty}$-diffeomorphisms of a segment.

\begin{thm}
Let $\varepsilon >0$. There exists $\varepsilon'>0$ such that any diffeomorphism in $\mathrm{Diff}^{\infty}([0,1])$ which is $\varepsilon'$-close to the identity and with ITI fixed points at $0$ and $1$ can be written as a product of $5$ $\varepsilon$-commutators.
\end{thm}

\section{Distortion and reducibility}
\label{s:dist-reduc}

%{\color{red} Il faudrait introduire la section et son plan... Est-ce que ça n'aurait pas un sens de d'abord énoncer le critère de distorsion de la prop 4.5 (précédé du Lemme 4.3 nécessaire à la définition des $a_{C,\eta}$), et seulement ensuite les outils ? L'enchaînement logique serait dans ce cas : 4.3, énoncé de 4.5, propriétés 4.4 et leur preuve, déduction du thm C et finalement preuve de 4.5 à partir de 4.1-4.2. Qu'en penses-tu?}
%
%{\color{blue}Je n'y suis pas vraiment favorable parce que, dans ma tête, le point de départ est vraiment la première proposition \ref{Prop:distseqint} que l'on veut appliquer coûte que coûte. Je fais une tentative d'explication.}
%

Let $G=\mathrm{Diff}^{\infty}_+(\mathbb{S}^1)$ or $\mathrm{Diff}^{\infty}_c(\mathbb{R})$.

In this section, we prove Theorem \ref{Thm:conjugacyanddistortion}. The starting point of the proof of the theorem is Proposition \ref{Prop:distseqint}, which states that, for any sequence of commutators of diffeomorphisms which are all supported in a common segment and converge sufficiently fast to $0$, all of these commutators belong to a finitely generated group of diffeomorphisms, with a control on the wordlength of each of those elements. Using the local fragmented perfection result from the previous section, we generalize this statement as Proposition \ref{Prop:distseqcircle} to any sequence of elements in $G$ which converge sufficiently fast to the identity. To prove that an element $f$ is distorted, we want to decompose its powers as products of elements which are close to the identity, and then apply Proposition \ref{Prop:distseqcircle} to the elements appearing in such a decomposition. This then gives an estimate on the wordlength of the powers of the element $f$ and allows to prove that the element is distorted. But to do that, we have to make sure that, for any sufficiently large $n$ and any neighbourhood $\mathcal{U}$ of the identity, the element $f^n$ can be written as a large product of a small number of elements of $\mathcal{U}$, number which can be bounded independently of $n$ and $\mathcal{U}$. This is what we achieve with Lemma~\ref{Lem:smallpieces} and the definitions of quantities $a_{C,\eta}$ which are used to characterize distortion elements (see Proposition \ref{Prop:characdistortion}). These quantities behave well with respect to conjugacy, which allows to finish the proof of Theorem \ref{Thm:conjugacyanddistortion}.

We fix a distance $d_{\infty}$ for the $C^{\infty}$-topology on $G$. Let $I$ be an open interval and $J$ be another open interval with $\overline{J} \subset I$.\medskip

As said earlier, the following proposition is the main tool to prove Theorem \ref{Thm:conjugacyanddistortion}. It relies on a trick used by Thurston to prove that groups of smooth diffeomorphisms of manifolds are perfect. Similar statements have already been used in the articles \cite{CF}, \cite{Av08},\cite{Mil}, \cite{Mil2} about distortion elements, among others.

\begin{prop} \label{Prop:distseqint}
There exists a sequence $(\varepsilon_n)_n$ of positive real numbers such that the following property holds. For any sequences $(g_n)_{n \geq 0}$ and $(g'_n)_{n \geq 0}$ of diffeomorphisms supported in $J$ with, for any $n$, $d_{\infty}(g_n,Id)<\varepsilon_n$ and $d_{\infty}(g'_n,Id)<\varepsilon_n$, there exists a finite subset $S$ of  $\mathrm{Diff}^{\infty}_c(I)$ such that, for any $n$
\begin{enumerate}
\item the commutator $[g_n,g'_n]$ belongs to $<S>$;
\item $\ell_{S}([g_n,g'_n]) \leq 14 n+14$.
\end{enumerate}
\end{prop}

\begin{proof}
Let $J'$ be an open interval with $\overline{J} \subset J' \subset \overline{J'} \subset I$.
Take a diffeomorphism $h$ in $\mathrm{Diff}^{\infty}_c(I)$ such that the intervals $h^{n}(J')$, for $n \geq 0$, are pairwise disjoint and the sequence of intervals $(h^{n}(J'))_{n \geq 0}$ converges to a point $p \in I$ for the Hausdorff topology. For any $n \geq 0$, take $\varepsilon_n>0$ sufficiently small so that, for any $g \in \mathrm{Diff}^{\infty}_c(J')$ with $d_{\infty}(g,Id) < \epsilon_n$, we have $$d_{\infty}(h^{n}gh^{-n},Id) < \frac{1}{n}.$$

Now, take sequences $(g_n)_{n \geq 0}$ and $(g'_n)_{n \geq 0}$ of diffeomorphisms which satisfy the hypothesis of Proposition \ref{Prop:distseqint}. We will construct a subset $S$ of  $\mathrm{Diff}^{\infty}_c(I)$ which satisfies the conclusions of Proposition \ref{Prop:distseqint}. Let $h'$ be a diffeomorphism in $\mathrm{Diff}_c(J')$ such that $h'(J)$ is disjoint from $J$. For any $n\geq 0$, we let $h'_{n}=h^{n}h' h^{-n}$. The diffeomorphism $h'_n$ is supported in $h^{n}(J')$ and sends the interval $h^{n}(J)$ to a disjoint interval.
Observe that, for any $k \neq l$, the diffeomorphisms $h^{k}g_k h^{-k}$ and $h^{l}g_l h^{-l}$ have disjoint support, namely the support of the first diffeomorphism is contained in $h^{k}(J)$ while the support of the second diffeomorphism is contained in $h^{l}(J)$. Moreover, as the sequence of diffeomorphisms $(h^{n}g_n h^{-n})_n$ converges to the identity in the $C^{\infty}$ sense, the following diffeomorphism
$$F= \prod_{n=0}^{\infty}h^{n}g_n h^{-n}$$
is infinitely differentiable at $p$ and defines a $C^\infty$ diffeomorphism supported in $I$. In the same way, one can define 
$$F'= \prod_{n=0}^{\infty}h^{n}g'_n h^{-n}.$$
We set $S=\left\{ h,h',F,F' \right\}$.

For any $n \geq 0$, let us look at the diffeomorphism $A=F h'_{n}F^{-1}(h'_{n})^{-1}$. As the diffeomorphism $h'_n$ is supported in $h^{n}(J')$, this diffeomorphism is equal to the identity outside $h^{n}(J')$. On $h^n(J')$, this diffeomorphism is equal to $h^{n}g_n h^{-n}$ on $h^{n}(J)$, to $h^{n}h'g_n(h')^{-1}h^{-n}$ on $h^{n}(h'(J))$ and is equal to the identity elsewhere. In the same way, we define $B=A=F'h'_{n}F'^{-1}(h'_{n})^{-1}$ and $C=F^{-1}F'^{-1}h'_{n}F'F(h'_{n})^{-1}$. Then
$$ABC= h^{n}[g_n,g'_n]h^{-n}.$$
Hence $[g_n,g'_n]=h^{-n}ABCh^{n} \in <S>$ and
$$\begin{array}{rcl}
 \ell_S([g_n,g'_n]) & \leq & 2n \ell_{S}(h) +6 \ell_{S}(h'_{n})+4 \ell_{S}(F)+4 \ell_{S}(F') \\
  & \leq & 2n+6.(2n+1)+8=14n+14.
\end{array}$$
\end{proof}

The following proposition is a generalisation of Proposition \ref{Prop:distseqint} to sequences of elements of~$G$. It is a consequence of Proposition \ref{Prop:distseqint} and Theorem \ref{Thm:locfragperf} or Theorem~\ref{Thm:localperfectionR}.

\begin{prop} \label{Prop:distseqcircle}
There exists a sequence of positive numbers $(\eta_n)_{n \geq 0}$ such that, for any sequence of diffeomorphisms $(f_n)_{n \geq 0}$ in $G$ with a common compact support and $d_{\infty}(f_n,Id) < \eta_n$, the following property holds: there exists a finite subset $S$ of  $G$ such that, for any $n$, the diffeomorphism $f_n$ belongs to $<S>$ and
$$\ell_S(f_n) \leq 70n+70.$$
\end{prop}

\begin{proof}

Let us prove first this theorem in the case where $G=\mathrm{Diff}^{\infty}_+(\mathbb{S}^1)$. Let $(\varepsilon_n)_n$ be the sequence given by Proposition \ref{Prop:distseqint}.
For each $n$, we apply Theorem \ref{Thm:locfragperf} with $\varepsilon=\varepsilon_n$, which gives a sequence $(\eta_n)_{n \geq 0}$ of positive real numbers. Now, take a sequence $(f_n)_n$ of diffeomorphisms which satisfies the hypothesis of Proposition~\ref{Prop:distseqcircle}. By Theorem \ref{Thm:locfragperf}, for any integer $n$, there exist two families of diffeomorphisms $(g_{i,n})_{1 \leq i \leq 4}$ and $(g'_{i,n})_{1 \leq i \leq 4}$ such that $g_{1,n},g'_{1,n},g_{4,n},g'_{4,n}$ are supported in $J_1$ while $g_{2,n},g'_{2,n},g_{3,n},g'_{3,n}$ are supported in $J_2$ with $$f_n=[g_{1,n},g'_{1,n}][g_{2,n},g'_{2,n}][g_{3,n},g'_{3,n}][g_{4,n},g'_{4,n}].$$

Moreover, each $g_{i,n}$ and $g'_{i,n}$ is $\varepsilon_n$-close to the identity. Hence, for any $1 \leq i \leq 4$, we can apply Proposition \ref{Prop:distseqint} to the sequence $([g_{i,n},g'_{i,n}])_{n \geq 0}$, to obtain a finite set $S_i$ of diffeomorphisms such that, for any $n$, the element $[g_{i,n},g'_{i,n}]$ belongs to the subgroup generated by $S_i$ and $\ell_{S_i}([g_{i,n},g'_{i,n}]) \leq 14n +14$. Take $S=S_1 \cup S_2 \cup S_3 \cup S_4$. Then, for any $n$, $f_n \in <S>$ and
$$\begin{array}{rcl}
\ell_S(f_n) & \leq & \displaystyle \sum_{i=1}^4 \ell_S([g_{i,n},g'_{i,n}]) \\
& \leq & 4(14n+14)=56n+56.
 \end{array}$$

The proof in the case where $G=\mathrm{Diff}^{\infty}_c(\mathbb{R})$ is similar, using Theorem \ref{Thm:localperfectionR} instead of Theorem~\ref{Thm:locfragperf}. As any small element is a product of $5$ commutators, we end up with an upper bound which is $70n +70$ instead of $56n +56$. Also, all the diffeomorphisms $f_n$ have to be supported in a common interval $I$ so that they are commutators of small elements supported in an open interval $J$ with $\overline{I} \subset J$. This last result is obtained by identifying $J$ with $\mathbb{R}$ and by applying Theorem \ref{Thm:localperfectionR}. We apply then Proposition \ref{Prop:distseqint} to those commutators of diffeomorphisms to finish the proof.
\end{proof}

To use Proposition \ref{Prop:distseqcircle} to prove that an element $f$ of $G$ is distorted, we need to decompose the powers of $f$ as a product of a uniformly bounded number of elements close to the identity. To achieve this, we start with the following lemma.

\begin{lem} \label{Lem:smallpieces}
There exists a constant $C_0>0$ such that, for any diffeomorphism $f$ in $G$ and any $\eta>0$, there exists a finite subset $S$ of $G$, with $|S| \leq C_0$ with the following properties.
\begin{enumerate}
\item The subset $S$ consists of diffeomorphisms which are $\eta$-close to the identity.
\item The diffeomorphism $f$ belongs to the subgroup generated by $S$.
\end{enumerate}
\end{lem}

%{\color{red} Pourquoi on ne dit pas simplement : ``Let $\eta>0$. There exists a constant $C_0>0$ such that any $f\in G$ can be written as a product of at most $C_0$ elements of $G$ each $\eta$-close to the identity.''~? }

\begin{proof}
Observe that, if $f$ is the time-one map of a $C^{\infty}$ flow $(\varphi^{t})_{t \geq 0}$ on the circle then we can write $f= \left( \varphi^{\frac{1}{n}} \right)^{n}$, with sufficiently large $n$, so that 
%the lemma holds 
properties 1 and 2 above hold by taking $S= \left\{\varphi^{\frac{1}{n}} \right\}$.

Recall that the group $G$ is a simple group. As a conjugate of a flow by a diffeomorphism is still a flow, the subgroup of $G$ generated be time-one maps of flows is a nontrivial normal subgroup of $G$. Hence any diffeomorphism in $G$ can be written as a product of time-one maps of flows.

For any diffeomorphism $f$ in $G$, we define $\nu(f)$ as the minimal number of time-one maps of flows required to write $f$ as a product of time-one maps of flows. Observe that, for any diffeomorphisms $f$ and $g$, $\nu(fg) \leq  \nu(f)+\nu(g)$, $\nu(gfg^{-1})=\nu(f)$, $\nu(f^{-1})=\nu(f)$ and $\nu(f)=0$ if and only if $f=Id$. Hence the map $\nu$ is a conjugation-invariant norm on the group $G$ in the sense of the article \cite{MR2509711} by Burago, Ivanov and Polterovich. By Theorem 1.11 of this article in the case where $G=\mathrm{Diff}_+^{\infty}(\mathbb{S}^1)$ and by Theorem 1.17 in the case where $G=\mathrm{Diff}_c^{\infty}(\mathbb{R})$, such a norm $\nu$ has to be bounded above by some constant $C_0>0$.
Hence, for any diffeomorphism $f$ in $\mathrm{Diff}^{\infty}_+(\mathbb{S}^1)$, there exists $k \leq C_0$ and $C^{\infty}$ flows $(\varphi^{t}_i)_{t \in \mathbb{R}}$, for $1 \leq i \leq k$ such that
$$f =\varphi_1^1 \varphi_2^1 \ldots \varphi_k^{1}.$$
Take $n$ sufficiently large so that each diffeomorphism $\varphi^{\frac{1}{n}}_i$ is $\eta$-close to the identity. Then  
$$f =\left(\varphi_1^{\frac{1}{n}}\right)^n \left(\varphi_2^{\frac{1}{n}}\right)^n \ldots \left(\varphi_k^{\frac{1}{n}}\right)^n.$$

Then Properties 1 and 2 of the lemma hold with 
$$S= \left\{ \varphi_i^{\frac{1}{n}} \ | \ 1 \leq i \leq k \right\}.$$
Moreover $|S| \leq k \leq C_0$.

By looking closely at the proofs in \cite{MR2509711}, we can see that we can take $C_0 = 16$ when $G=\mathrm{Diff}_+^{\infty}(\mathbb{S}^1)$ and $C_0=14$ when $G=\mathrm{Diff}_c^{\infty}(\mathbb{R})$.

\end{proof}

Take $C \geq C_0$ and $\eta>0$. Fix a diffeomorphism $f$ in $G$. By Lemma \ref{Lem:smallpieces}, there exists a subset $S$ of $G$ which consists of diffeomorphisms which are $\eta$-close to the identity with $|S| \leq C$ such that $f \in <S>$. We will call a $(\eta,C)$-decomposition of $f$ any decomposition of $f$ as a product of elements of $S \cup S^{-1}$ for such an $S$. Now, let $A_{C,\eta}(f)$ be the infimum of $\ell_S(f)$ over all such subsets $S$. We will call a minimal $(\eta,C)$-decomposition of $f$ any $(\eta,C)$-decomposition of $f$ of minimal word length.

Finally, we define
$$a_{C,\eta}(f)=\limsup_{n \rightarrow +\infty} \frac{A_{C,\eta}(f^{n})}{n}.$$

\begin{prop}[Some properties of $A$ and $a$] \label{Prop:properties}
Fix $C \geq C_0$ and $\eta>0$. For any $f,g,h \in G$ and any $n>0$
\begin{enumerate}
\item $A_{2C,\eta}(fg) \leq A_{C,\eta}(f)+A_{C,\eta}(g).$
\item $A_{C,\eta}(f^{n}) \leq n A_{C,\eta}(f)$.
\item $A_{2C,\eta}(hfh^{-1}) \leq A_{C,\eta}(f)+ 2A_{C,\eta}(h)$.
\item  $\displaystyle a_{2C,\eta}(f) \leq \frac{a_{C,\eta}(f^{n})}{n}$.
\item $a_{2C,\eta}(f) \leq a_{C,\eta}(hfh^{-1})$.
\end{enumerate}
\end{prop}

\begin{proof}
The first item is proved by concatenating minimal $(\eta,C)$-decompositions for $f$ and for $g$. The second one is obtained by concatenating $n$ times a minimal $(\eta,C)$-decomposition for $f$. A minimal $(\eta,C)$-decomposition for $h$ yields a minimal $(\eta,C)$-decomposition for $h^{-1}$ with the same generating set, by taking inverses. The third item is then proved by concatenating those decompositions with a minimal $(\eta,C)$-decomposition for $f$.

Let us prove the fourth point. For any integer $m>0$, we write $m=q_m n+r_m$ the euclidean division of $m$ by $n$. Then, using the two first items,
$$\begin{array}{rcl}
A_{2C,\eta}(f^{m}) & \leq & A_{C,\eta}(f^{q_m n })+A_{C,\eta}(f^{r_m }) \\
 & \leq  & q_{m}A_{C,\eta}(f^{ n })+A_{C,\eta}(f^{r_m }).
 \end{array}$$
Dividing by $m$ and taking the limit superior gives the wanted inequality.

Finally, for any $n>0$, by the third point,
$$ A_{2C,\eta}(f^n) \leq A_{C,\eta}(hf^nh^{-1})+ 2A_{C,\eta}(h).$$
Dividing by $n$ and taking the limit superior gives the wanted inequality.

\end{proof}

The following proposition is the reason why we are interested in the quantities $a_{C,\eta}$.

\begin{prop} \label{Prop:characdistortion}
An element $f$ of $G$ is distorted if and only if there exists $C>0$ such that, for any $\eta>0$,
$$a_{C,\eta}(f)=0.$$
\end{prop}

\begin{proof}
Suppose that $f$ is distorted. Let $S$ be a finite set such that $f \in <S>$ and 
$$\lim_{n \rightarrow +\infty} \frac{\ell_S(f^{n})}{n}=0.$$
Fix $\eta>0$. Let $M_{\eta}= \max_{s \in S} A_{C_0,\eta}(s)$. Then, for any $n>0$, by concatenating minimal $(C_0,\eta)$-decompositions of each element of $S$, we obtain
$$\frac{A_{C_0 |S|,\eta}(f^{n})}{n} \leq \frac{\ell_S(f^{n})}{n}M.$$
Taking the limit as $n$ goes to $+\infty$, we obtain that $a_{C_0|S|,\eta}(f)=0$.\medskip

Conversely, suppose that there exists $C>0$ such that, for any $\eta>0$,
$$a_{C,\eta}(f)=0.$$
Fix an integer $k \geq 1$. Let 
$$\lambda_k= \min_{n \leq kC-1}  \eta_{n},$$
where the sequence $(\eta_n)_{n \geq 0}$ is given by Proposition \ref{Prop:distseqcircle}.
As $a_{C,\lambda_k}(f)=0$, we can find an integer $n_{k}$ such that
$$\frac{A_{C,\lambda_k}(f^{n_k})}{n_k} (70Ck+70) \leq \frac{1}{k}.$$
We denote by $S_k$ a finite subset of $G$ with $|S_k| \leq C$ such that a decomposition of $f^{n_k}$ into a product of elements of $S_k \cup S_k^{-1}$ is a minimal $(\lambda_k,C)$-decomposition for $f^{n_k}$.

Now, we define a sequence $(f_n)_{n \geq 0}$ of elements of $G$ by listing successively the elements of the sets $S_k$, for $k \geq 1$. By the choice of $(\lambda_k)_k$, the diffeomorphism $f_n$ is $\eta_n$-close to the identity. Proposition \ref{Prop:distseqcircle} then gives a finite subset $S'$ of $G$ such that all the elements of the $S_k$s belong to the group generated by $S'$ and elements of $S_k$ have a word length smaller than $70Ck+70$. Let $S=S'\cup \left\{ f \right\}$. Then, by using the minimal $(\lambda_k,C)$-decomposition of $f^{n_k}$ into elements of $S_{k}$ and the word length estimate of the elements of $S_k$, we obtain

$$\begin{array}{rcl}
\displaystyle \frac{\ell_{S}(f^{n_k})}{n_k}  & \leq & \displaystyle   \frac{A_{C,\lambda_k}(f^{n_k})}{n_k} (70Ck+70) \\
 
 & \leq & \displaystyle \frac{1}{k}.
 \end{array}
 $$
Hence $\displaystyle \frac{l_{S}(f^{n_k})}{n_k} \rightarrow 0$ and $f$ is distorted.

\end{proof}

Proposition \ref{Prop:characdistortion} allows us to finish the proof of Theorem \ref{Thm:conjugacyanddistortion}.

\begin{proof}[Proof of Theorem \ref{Thm:conjugacyanddistortion}]
By Proposition \ref{Prop:characdistortion}, as the element $g$ is distorted in $G$, there exists $C>0$ such that, for any $\eta>0$, $a_{C,\eta}(g)=0$. Fix $\eta>0$.
By Proposition \ref{Prop:characdistortion}, it suffices to prove that $a_{8C,\eta}(f)=0$. Fix $n>0$. Take an integer $m>0$ such that $h_{m}f^nh_{m}^{-1}g^{-n}$ is $\eta$-close to the identity. By the fourth item of Proposition \ref{Prop:properties}
$$a_{8C,\eta}(f) \leq \frac{a_{4C,\eta}(f^{n})}{n}.$$
Using the fifth item of Proposition \ref{Prop:properties}, we obtain
$$ \frac{a_{4C,\eta}(f^{n})}{n} \leq \frac{a_{2C,\eta}(h_{m}f^{n}h_{m}^{-1})}{n}. $$

However, by the first item of Proposition \ref{Prop:properties},
$$A_{2C,\eta}(h_{m}f^{n}h_{m}^{-1}) \leq A_{C,\eta}(h_{m}f^{n}h_{m}^{-1}g^{-n})+A_{C,\eta}(g^{n}) \leq 1+ A_{C,\eta}(g^{n})$$
and, by the second item of Proposition \ref{Prop:properties}, 
$$a_{2C,\eta}(h_{m}f^{n}h_{m}^{-1}) \leq A_{2C,\eta}(h_{m}f^{n}h_{m}^{-1})$$
so that
$$ \frac{a_{4C,\eta}(f^{n})}{n} \leq \frac{1}{n}+\frac{A_{C,\eta}(g^{n})}{n}.$$
Therefore
$$a_{8C,\eta}(f)\leq  \frac{1}{n}+\frac{A_{C,\eta}(g^{n})}{n}.$$
Taking the limit superior of the right-hand side, we obtain
$$a_{8C,\eta}(f) \leq a_{C,\eta}(g)=0.$$
\end{proof}

\section{Extension of Theorem \ref{Thm:conjugacyanddistortion} to any manifold}
\label{s:extension}

Let $G$ be either the identity component $\mathrm{Diff}_0^{\infty}(M)$ of the group of $C^{\infty}$-diffeomorphisms of a closed manifold $M$ of dimension at least $2$ or the group $\mathrm{Diff}_c^{\infty}(\mathbb{R}^d)$ of compactly supported $C^{\infty}$-diffeomorphisms of $\mathbb{R}^d$, with $d \geq 2$. In this last case, we let $M=\mathbb{R}^d$.

In this section, we explain why Theorem \ref{Thm:conjugacyanddistortion} extends to such a group. The structure of the proof is the same as in the $1$-dimensional case but other existing results are required in this case to achieve the proof of the theorem. Below, we explain how to prove the analogs of Propositions \ref{Prop:distseqint}, \ref{Prop:distseqcircle}, Lemma \ref{Lem:smallpieces}, Proposition \ref{Prop:properties} and Proposition \ref{Prop:characdistortion} in this general case.

With $I$ and $J$ being open balls instead of open intervals, Proposition \ref{Prop:distseqint} and its proof are still valid. 
%{\color{red} and standard? ref? (mêmes questions pour d'autres lemmes de cette partie)} {\color{blue} En fait tu répètes vraiment mot pour mot la démo en remplaçant les intervalles par des boules. Je pensais que dire "its proof is valid" était suffisant. C'est pareil pour les autres lemmes où je ne donne pas de détail.}

The statement of Proposition \ref{Prop:distseqcircle} is still valid for those more general groups with one little change: the constants $70$ must be replaced by constants which depend on the manifold. To be precise, those $70$ have to be changed into $28 \mathrm{dim}(M) \mathcal{C}(M)$, where $\mathcal{C}(M)$ is the minimal number of embedded closed balls in $M$ whose interiors cover $M$ if $M \neq \mathbb{R}^d$ and is equal to $1$ if $M = \mathbb{R}^d$. 

The proof of Proposition \ref{Prop:distseqcircle} is the same but one needs to use the local perfection results adapted to higher dimensional manifolds. Namely, we use the two following results.

Suppose $M \neq \mathbb{R}^d$ and fix a covering $\mathcal{U}=(U_i)_{1 \leq i \leq \mathcal{C}}$ of $M$ by interiors of closed balls, with $\mathcal{C}=\mathcal{C}(M)$. The following lemma is classical (see Theorem 2.2.1 of \cite{Bou} for a proof).

\begin{lem}[Fragmentation lemma] \label{l:frag}
Let $\varepsilon >0$. There exists $\eta>0$ such that, for any diffeomorphism $f$ in $\mathrm{Diff}_0^{\infty}(M)$ with $d_{\infty}(f,Id_M) < \eta$, there exists a family of diffeomorphisms $(f_i)_{1 \leq i \leq \mathcal{C}}$ such that, for each $i$, $f_i$ is supported in $U_i$ and is $\varepsilon$-close to the identity, and 
$$f=f_1\circ f_2 \circ \ldots \circ f_{\mathcal{C}}.$$
\end{lem} 
 
To get the local perfection result we need, we combine this lemma with the following theorem by Haller, Rybicki and Teichmann (see \cite{HRT} and also \cite{Mann} and the appendix of \cite{Mil} for more elementary realisations of the ideas of Haller, Rybicki and Teichmann).

\begin{thm}[Local perfection for diffeomorphisms supported in balls] \label{t:localperfhigher}
Let $\varepsilon >0$. There exists $\eta>0$ such that, for any diffeomorphism $f$ in $\mathrm{Diff}_c^{\infty}(\mathbb{R}^d)$ which is supported in the open unit ball $B$ with $d_{\infty}(f,Id) < \eta$, there exists a family of diffeomorphisms $(f_{i,j})_{1 \leq i \leq d, 1 \leq j \leq 4}$ in $\mathrm{Diff}^{\infty}_c(\mathbb{R}^d)$ which are $\varepsilon$-close to the identity such that
$$f=[f_{1,1},f_{1,2}][f_{1,3},f_{1,4}][f_{2,1},f_{2,2}][f_{2,3},f_{2,4}] \ldots [f_{d,1},f_{d,2}][f_{d,3},f_{d,4}].$$ 
\end{thm}

Combining Lemma \ref{l:frag} and Theorem \ref{t:localperfhigher}, we obtain that any diffeomorphism in $G$ sufficiently close to the identity can be written as a product of $2 \mathrm{dim}(M) \mathcal{C}(M)$ commutators of elements close to the identity which are supported in one of the balls of $\mathcal{U}$ in the case where $M \neq \mathbb{R}^d$. In the case where $M =\mathbb{R}^d$, it suffices to use Theorem \ref{t:localperfhigher} directly. This allows to prove the analog of Proposition \ref{Prop:distseqcircle} in this more general context.

The statement of Lemma \ref{Lem:smallpieces} is the same in this more general context but the proof, given below, is more involved.
%its proof is more involved in this more general context. We prove it below.

The definition of the quantities $A_{C,\eta}$, $a_{C,\eta}$, the statement and the proofs of Propositions \ref{Prop:properties} and \ref{Prop:characdistortion}, and the end of the proof of Theorem \ref{Thm:conjugacyanddistortion} are exactly the same in this more general case.
 
\begin{proof}[Proof of Lemma \ref{Lem:smallpieces} in the general case]
We first treat the case where $G=\mathrm{Diff}_c^{\infty}(\mathbb{R}^d)$, with $d\geq2$. As in the one-dimensional case, Theorem 1.17 from the article \cite{MR2509711} by Burago, Ivanov and Polterovich implies that there exists $C(d)>0$ such that any diffeomorphism $f \in G$ can be written as a product of at most $C(d)$ time-one maps of compactly supported flows. This allows to conclude the proof of Lemma \ref{Lem:smallpieces} in the same way as in the one-dimensional case, with $C_0=C(d)$.

In the case of a compact manifold $M$ with dimension $d \geq 2$, we will need the following lemma

\begin{lem} \label{l:reductiongeneratingset}
Let $(f_n)_{0\leq n \leq \ell}$ be a finite sequence of diffeomorphisms in $\mathrm{Diff}^{\infty}_c(\mathbb{R}^d)$. Then there exists a finite subset $S \subset \mathrm{Diff}^{\infty}_c(\mathbb{R}^d)$, with $|S| = 4$ such that, for any $n$, the diffeomorphism $f_n$ belongs to the group generated by $S$. 
\end{lem}

Let us explain how to finish the proof of Lemma \ref{Lem:smallpieces} for a general compact manifold $M$ using Lemma \ref{l:reductiongeneratingset}, before concluding with the proof of Lemma \ref{l:reductiongeneratingset}. Fix a finite cover $\mathcal{U}$ of $M$ by open balls, a small real number $\eta>0$ and a diffeomorphism $f$ in $\mathrm{Diff}^{\infty}_0(M)$. By the fragmentation lemma \ref{l:frag} and the connectedness of $\mathrm{Diff}_0^{\infty}(M)$, there exists a finite sequence $(f_i)_{1 \leq i \leq \ell}$ of smooth diffeomorphisms of $M$ such that each $f_i$ is supported in one of the balls of $\mathcal{U}$ and
$$f=f_1 \circ f_2 \circ \ldots \circ f_{\ell}.$$
For each open ball $U \in \mathcal{U}$, denote by $(g_{j,U})_j$ the finite sequence consisting of the diffeomorphisms $f_i$ which are supported in the open set $\mathcal{U}$. Applying Lemma \ref{l:reductiongeneratingset} to this finite sequence yields a finite subset $S_U$ of diffeomorphisms with cardinality at most $4$ such that all the diffeomorphisms $g_{j,U}$ belong to the group generated by $S_U$. Now, as each open set $U$ in $\mathcal{U}$ is diffeomorphic to $\mathbb{R}^d$, we can apply Lemma \ref{Lem:smallpieces} to each diffeomorphism in the union of the $S_U$, with $U \in \mathcal{U}$. This gives a finite subset $S$ of $\mathrm{Diff}_0^{\infty}(M)$ with cardinality at most $4 C(d) |\mathcal{U}|$ all of whose elements are $\eta$-close to the identity and such that the diffeomorphism $f$ belongs to the group generated by $S$.
\end{proof}

\begin{proof}[Proof of Lemma \ref{l:reductiongeneratingset}]
The proof is very similar to the proof of Proposition \ref{Prop:distseqint} so we will give less details here. Fix a sequence $(f_n)_{0 \leq n \leq \ell}$ of diffeomorphisms in $\mathrm{Diff}_c^{\infty}(\mathbb{R}^d)$. By Theorem \ref{t:localperfhigher} and the connectedness of $\mathrm{Diff}_c^{\infty}(\mathbb{R}^d)$, each diffeomorphism $f_n$ can be written as a product of commutators of diffeomorphisms in $\mathrm{Diff}^{\infty}(\mathbb{R}^d)$. Denote by $([g_n,h_n])_{0 \leq n\leq k}$ the finite sequence of commutators appearing in one of those decompositions and notice that it suffices to prove Lemma \ref{l:reductiongeneratingset} for this sequence of commutators.

Take $R>0$ such that the support of all the diffeomorphisms $g_n$ and $h_n$ is contained in the open ball $B(0,R)$ of radius $R$ and take $R'>R$. Take a diffeomorphism $t$ in $\mathrm{Diff}^{\infty}_c(\mathbb{R}^d)$ such that the balls $t^{n}(B(0,R'))$, for $n \geq 0$, are pairwise disjoint. Take also a diffeomorphism $t_0$ supported in $B(0,R')$ such that $t_0(B(0,R)) \cap B(0,R)= \emptyset$. Finally, let
$$G=\prod_{i=0}^k t^i g_it^{-i}$$
and
$$H=\prod_{i=0}^k t^i h_i t^{-i}.$$
We let $S=\left\{t,t_0,G,H \right\}$ and we will prove that each commutator $[g_n,h_n]$ belongs to the group generated by $S$. For $i \geq 0$, observe that the diffeomorphism $t_i=t^it_0t^{-i}$ is supported in $t^{i}(B(0,R'))$ and displaces the ball $t^{i}(B(0,R))$. Then, for any $n \leq k$,
$$t^{n}[g_n,h_n]t^{-n}= G t_n G^{-1} t_{n}^{-1}H t_n H^{-1} t_{n}^{-1} G^{-1}H^{-1} t_n H G t_{n}^{-1},$$
which proves the lemma.
\end{proof}

\section*{Acknowledgements, fundings}

We are grateful to the organizers of the IHP trimester \emph{Group Actions and Rigidity: Around the Zimmer Program}, to the IHP itself (UAR 839 CNRS-Sorbonne Université), and to LabEx CARMIN (ANR-10-LABX-59-01), for providing the stimulating environment in which this work was initiated. We would also like to thank Andrés Navas for inspiring discussions and useful comments. The interaction between the two authors was also boosted by the ANR project \emph{GROMEOV} (2019-2025) ANR-19-CE40-0007. Supplementary funding was provided by ECOS Grant 23003.

%The proof of Lemma \ref{Lem:smallpieces} follows the same idea : one has to write any diffeomorphism as a product of a bounded number of time $1$ maps of flows, with a bound depending only on the manifold. In the case where $M=\mathbb{R}^d$, this is a consequence of Theorem 1.17 in \cite{MR2509711}. Case where $M$ is $3$-dimensional and of the $2$-sphere (BIP). Case where $M$ is odd-dimensional with dimension $\geq 5$ or $M$ is even dimensional with no middle-index handlebody (Tsuboi 1). Case where $M$ is even-dimensional with dimension $\geq 6$ (Tsuboi 2). Explanation why there is no hope to prove Lemma \ref{Lem:smallpieces} with the same strategy in dimension $2$.

%%%%%%%%%%%%%%%%%%%%%%%%%%%%%%%%%%%%%%%%%%%%%%%%%%%

\begin{footnotesize}

\bibliographystyle{amsalpha}
\bibliography{Biblio}
\end{footnotesize}

\end{document}